\documentclass[11pt,twoside,leqno]{article}
 \pdfoutput=1
\usepackage[utf8]{inputenc}
\usepackage[T1]{fontenc}
\usepackage[english]{babel}
\usepackage{csquotes}
\usepackage{palatino}
\usepackage[a4paper,margin=1in,marginpar=1in]{geometry}
\usepackage{multirow}

\usepackage[shortlabels]{enumitem}

\usepackage[nottoc,numbib]{tocbibind} %

\newlist{condition}{enumerate}{10}
\setlist[condition]{label*={(\roman*).},ref={(\roman*)}}
\newlist{conditionalt}{enumerate}{10}
\setlist[conditionalt]{label*={(\roman*').},ref={(\roman*')}}

\usepackage[backend=biber,giveninits=true,maxnames=999,minnames=5]{biblatex}
\newcommand{\fcite}[1]{\citeauthor{#1}~\cite{#1}}
\usepackage{amssymb, amsmath, amsthm, amsfonts}
\usepackage{esint}
\usepackage{mathtools}

\usepackage[pdfpagelabels]{hyperref}
\usepackage{xcolor}
\definecolor{emerlandgreen}{HTML}{2ECC71}
\definecolor{peterriverblue}{HTML}{3498DB}
\definecolor{alizarinred}{HTML}{E74C3C}
\definecolor{contrastgrey}{HTML}{586E75}
\hypersetup{%
colorlinks,
linktocpage,
urlcolor  = peterriverblue,
citecolor = emerlandgreen,
linkcolor = alizarinred}

\usepackage[capitalise,nameinlink,sort]{cleveref}
\AfterEndEnvironment{proof}{\noindent\ignorespaces}
\makeatletter
\def\@endtheorem{\endtrivlist}
\makeatother
\Crefname{equation}{}{}
\Crefname{enumi}{}{}
\Crefname{conditioni}{Condition}{Conditions}
\Crefname{conditionalti}{Condition}{Conditions}
\newtheorem{theorem}{Theorem}[section]
\Crefname{theorem}{Theorem}{Theorems}
\newtheorem{lemma}[theorem]{Lemma}
\Crefname{lemma}{Lemma}{Lemmas}
\newtheorem{proposition}[theorem]{Proposition}
\Crefname{proposition}{Proposition}{Propositions}

\Crefname{corollary}{Corollary}{Corollaries}

\Crefname{conjecture}{Conjecture}{Conjectures}

\Crefname{assumption}{Assumption}{Assumptions}

\theoremstyle{definition}

\Crefname{definition}{Definition}{Definitions}

\Crefname{question}{Question}{Questions}

\theoremstyle{remark}
\newtheorem{remark}{Remark}
\Crefname{remark}{Remark}{Remarks}

\Crefname{example}{Example}{Examples}
\newtheorem*{example*}{Example}
\numberwithin{equation}{section}

\usepackage{calrsfs}
\DeclareMathAlphabet{\pazocal}{OMS}{zplm}{m}{n}

\usepackage{authblk}

\usepackage{xparse}
\DeclareDocumentCommand\tto{O{n} O{\infty} m}{\xrightarrow[{#1}\to{#2}]{#3}}

\newcommand{\CA}{\mathcal{A}}

\newcommand{\CC}{\mathcal{C}}
\newcommand{\CD}{\mathcal{D}}

\newcommand{\CF}{\mathcal{F}}
\newcommand{\CG}{\mathcal{G}}

\newcommand{\CI}{\mathcal{I}}

\newcommand{\CL}{\mathcal{L}}
\newcommand{\CM}{\mathcal{M}}

\newcommand{\lN}{\mathbf{N}}

\newcommand{\lP}{\mathbf{P}}

\newcommand{\lo}{\mathbf{o}}

\newcommand{\gM}{\mathfrak{M}}

\newcommand{\gO}{\mathfrak{O}}

\newcommand{\gW}{\mathfrak{W}}

\newcommand{\gZ}{\mathfrak{Z}}

\newcommand{\dE}{\mathbb{E}}

\newcommand{\dN}{\mathbb{N}}

\newcommand{\dP}{\mathbb{P}}

\newcommand{\dR}{\mathbb{R}}

\newcommand{\dV}{\mathbb{V}}

\newcommand{\dd}{\mathrm{d}}

\DeclareMathOperator{\Var}{{\dV}ar}

\DeclareMathOperator{\e}{e}

\DeclareMathOperator{\dom}{{\CD}om}
\DeclareMathOperator{\cov}{{\CC}ov}

\DeclareSourcemap{
  \maps[datatype=bibtex]{
    \map{
      \step[fieldsource=doi,final]
      \step[fieldset=url,null]
    }  
  }
}

\author{Ronan Herry\footnote{Permanent contact: \href{mailto:ronan.herry@live.fr}{\texttt{ronan.herry@live.fr}}. ORCid: \href{https://orcid.org/0000-0001-6313-1372}{0000-0001-6313-1372}.}}

\affil{Institut für Angewandte Mathematik, Universität Bonn.}
\title{Stable limit theorems on the Poisson space}

\begin{document}
\maketitle

\begin{abstract}
  We prove limit theorems for functionals of a Poisson point process using the Malliavin calculus on the Poisson space.
  The target distribution is conditionally either a Gaussian vector or a Poisson random variable.
  The convergence is stable and our conditions are expressed in terms of the Malliavin operators.
  For conditionally Gaussian limits, we also obtain quantitative bounds, given for the Wasserstein transport distance in the univariate case; and for another probabilistic variational distance in higher dimension.
  Our work generalizes several limit theorems on the Poisson space, including the seminal works by \fcite{PeccatiSoleTaqquUtzet} for Gaussian approximations; and by \fcite{PeccatiChenStein} for Poisson approximations; as well as the recently established fourth-moment theorem on the Poisson space of \fcite{DoeblerPeccatiFMT}.
  We give an application to stochastic processes.
\end{abstract}

\section*{Introduction}%
\label{section:intro}

One of the celebrated contributions of \citeauthor{RenyiStable} \cite{RenyiMixing,RenyiStable} is a refinement of the notion of convergence in law, commonly referred to as \emph{stable convergence}.
Stable convergence is tailored for studying conditional limits of sequences of random variables.
Thus, a stable limit is, typically, a \emph{mixture}, that is, in our terminology: a random variable whose law depends on a random parameter; for instance, a centered Gaussian random variable with random variance, or a Poisson random variable with random mean.
In the setting of semi-martingales, one book by \fcite{JacodShiryaev} summarizes archetypal stable convergence results involving such mixtures.
More recently, results by \fcite{NourdinNualartLimitsSkorokhod}; \fcite{HarnettNualartCLTStratanovich}; and \fcite{NourdinNualartPeccatiStableLimits} give sufficient conditions and quantitative bounds for the stable convergence of functionals of an isonormal Gaussian process to a \emph{Gaussian mixture}.
Typically, applications of such results study the limit of a sequence of quadratic functionals of a fractional Brownian motion.
The three references~\cite{NourdinNualartLimitsSkorokhod,HarnettNualartCLTStratanovich,NourdinNualartPeccatiStableLimits} make a pervasive use of the Malliavin calculus to prove such limit theorems.
Earlier works by \fcite{NulartOrtizLatorreCLTMalliavin} and by \fcite{NourdinPeccatiSteinWiener} initiate this approach: they use Malliavin calculus in order to prove central limit theorems for iterated Itô integrals initially obtained by \fcite{NualartPeccatiFourthMoment} with different tools.
These far-reaching contributions form a milestone in the theory of limit theorems and inaugurate an independent field of research, known as the \emph{Malliavin-Stein approach} (see the webpage of \fcite{MalliavinSteinWebpage} for a comprehensive list of contributions on the subject).

The trendsetting work of \fcite{PeccatiSoleTaqquUtzet} extends the Malliavin-Stein approach beyond the scope of Gaussian fields to Poisson point processes.
Despite being a very active field of research, the considered limit distributions are, most of the time, Gaussian~\cite{LachiezeReyPeccatiContractions,LachiezeReyPeccatiMarkedProcess,LastPeccatiSchulteSecondOrderPoincare,ReitznerSchulteCLTUStatistics,PeccatiReitzner,PeccatiZheng,SchulteCLTVoronoi,DoeblerPeccatiFMT,DoeblerVidottoZheng,BourguinPeccatiPortmanteau} or, sometimes, Poisson~\cite{PeccatiChenStein} or Gamma~\cite{PeccatiThaele}; to the best of our knowledge, prior to the present work, mixtures, have not been considered as limit distributions.
The aim of this paper is to tackle this problem, by proving an array of new quantitative and stable limit theorems on the Poisson space, with a target distribution given either by a Gaussian mixture, that is the distribution of a centered Gaussian variable with random covariance; or a Poisson mixture, that is the distribution of a Poisson variable with random mean.
We rely on two standard techniques to obtain our limit theorems: the \emph{characteristic functional} method, to obtain qualitative results; and an interpolation approach, known as \emph{smart path}, for the quantitative results.
In the two cases, we build upon various tools from stochastic analysis for Poisson point processes, such as the Malliavin calculus, integration by parts for Poisson functionals, and a representation of the carré du champ associated to the generator of the Ornstein-Uhlenbeck semi-group on the Poisson space.
Provided mild regularity assumptions on the functional under study, our approach allows us to deal, in \cref{theorem:convergence_stable_quali,theorem:convergence_stable_quantitative}, with any target distribution of the form $SN$, where $S$ is a matrix-valued random variable (measurable with respect to the underlying Poisson point process) and $N$ is a Gaussian vector independent of the underlying Poisson point process.
In the same way, in \cref{theorem:convergence_stable_quali_P}, we can consider any target distribution of the form of a Poisson mixture, whose precise definition is given below.

Let us now give a more detailed sample of the main results.
Throughout the paper, we study the asymptotic behaviour of a sequence $\{F_{n} = f_{n}(\eta)\}$ of square-integrable functionals of a Poisson point process $\eta$.
Here, $\eta$ is a Poisson point process on an arbitrary $\sigma$-finite measured space $(Z, \gZ, \nu)$ (for the moment, we simply recall that $\eta$ is a random integer-valued measure on $Z$ satisfying some strong independence properties and such that $\dE \eta = \nu$).
Moreover, we assume that the $F_{n}$'s are of the form $F_{n} = \delta u_{n}$, where $\delta$ is the Kabanov stochastic integral and $u_{n} = \{ u_{n}(z);\; z \in Z\}$ is a random function on $Z$ (for the moment, one can think of the slightly abusive definition of $\delta$ as the following pathwise stochastic integral $\delta u = \int u(z) (\eta - \nu)(\dd z)$).
As we will see, assuming that $F_{n} = \delta u_{n}$ is not restrictive, as, provided $\dE F_{n} = 0$, this equation always admits infinitely many solutions.
An important object in our study is the \emph{Malliavin derivative} of $F_{n}$ given by $D_{z} F_{n} = f_{n}(\eta + \delta_{z}) - f_{n}(\eta)$.
The crucial tool to establish our results is a duality relation (also referred to as integration by parts) between the operators $D$ and $\delta$: $\dE F \delta u = \dE \nu(u DF)$.
This relation is at the heart of the Malliavin-Stein approach to obtain limit theorems both in a Gaussian \cite[Chapter 5]{NourdinPeccatiBlueBook} and in a Poisson setting \cite{PeccatiSoleTaqquUtzet}.
For instance, we have the following result in our Poisson setting.
\begin{theorem}[{\cite[Theorem 3.1]{PeccatiSoleTaqquUtzet}}]\label{theorem:PSTU}
  Let the previous notation prevails, and assume that:
\begin{equation}\label{condition:sigma_2_deterministe}
  \nu(u_{n} DF_{n}) \tto{\CL^{1}(\dP)} \sigma^{2},
\end{equation}
and
\begin{equation}\label{condition:R3_0}
  \dE \int |u_{n}(z)| {|D_{z} F_{n}|}^{2} \nu(\dd z) \tto{} 0.
\end{equation}
Then\footnote{To be precise, the theorem of \cite{PeccatiSoleTaqquUtzet} chooses one particular solution of $F_{n} = \delta u_{n}$ but we do not enter into too many technical details in this introduction.}, we have that $F_{n} \tto{law} \lN(0,\sigma^{2})$.
\end{theorem}
By integration by parts, we see that $\dE \nu(u_{n} DF_{n}) = \dE F_{n}^{2}$ and, at the heuristic level, the quantity $\nu(u_{n} DF_{n})$ controls the asymptotic variance of $F_{n}$.
The condition \cref{condition:R3_0} arises from the non-diffusive nature of the Poisson process.
Following our heuristic, it is very natural to ask what happens to the conclusions of \cref{theorem:PSTU} when $\nu(u_{n} DF_{n})$ converges to a non-negative random variable $S^{2}$.
\cref{theorem:convergence_stable_quali} states that, in this case, provided \cref{condition:R3_0} and a condition of asymptotic independence hold, $(F_{n})$ converges stably to the Gaussian mixture $\lN(0,S^{2})$.
In fact, in \cref{theorem:convergence_stable_quali}, we are also able to deal with vector-valued random variables.
In the same fashion, \cref{theorem:convergence_stable_quali_P} gives sufficient conditions involving $u_{n}$ and $D F_{n}$ to ensure the convergence of $(F_{n})$ to a Poisson mixture (thus generalizing a result by \fcite{PeccatiChenStein} for convergence to Poisson random variables).
When targeting Gaussian mixtures, we are also able to provide quantitative bounds in a variational distance between probability laws (\cref{theorem:convergence_stable_quantitative} for the multivariate case, and \cref{theorem:stable_convergence_wasserstein} for the univariate case).

Following a recent contribution by \fcite{DoeblerPeccatiFMT}, we derive from our analysis a stable fourth moment theorem: a sequence of iterated Itô-Poisson integrals converges stably to a Gaussian (with deterministic variance) if and only if its second and fourth moment converge to those of Gaussian (\cref{proposition:stable_fourth_moment}).
For the limit of a sequence of order $2$ Itô-Poisson stochastic integrals to be a Gaussian or Poisson mixture, we obtain sufficient conditions  expressed in terms of analytical conditions on the integrands (\cref{theorem:order_2_gaussian,theorem:order_2_poisson}).
We also apply our results to study the limit of a sequence of quadratic functionals of a rescaled Poisson process on the line (\cref{theorem:quadratic}); hence, adapting to the Poisson setting a theorem of \fcite{PeccatiYor} for a standard Brownian motion (generalized by \cite{NourdinNualartPeccatiStableLimits} to the setting of a sufficiently regular fractional Brownian motion using Malliavin-Stein techniques; and generalized to any fractional Brownian motion by \cite{PratelliRigo} using ad-hoc computations).
 
The paper is organized as follows.
\cref{section:recollection} fixes the notations for the rest of the paper; recalls the definitions of probabilistic distances and of the Poisson point process; gives more information on Gaussian and Poisson mixtures that serve as target distributions in our limit theorems; and gives a brief review on stochastic analysis for Poisson point processes with a focus Malliavin operators that are at the hearth of our method.
We present in~\cref{section:abstract_results} the main results of this paper: \cref{theorem:convergence_stable_quali,theorem:convergence_stable_quali_P,theorem:convergence_stable_quantitative}, they contain stable and quantitative limit theorems for Poisson functionals.
A detailed comparison of these results with the aforementioned works on the Gaussian space of~\cite{NourdinNualartLimitsSkorokhod,HarnettNualartCLTStratanovich,NourdinNualartPeccatiStableLimits}, as well as with limit theorems on the Poisson space \cite{LachiezeReyPeccatiContractions,LachiezeReyPeccatiMarkedProcess,PeccatiSoleTaqquUtzet,PeccatiChenStein}, follows in~\cref{section:comparison}.
All the proofs are postponed to \cref{section:proofs}.
Special attention is paid to stochastic integrals in~\cref{section:stochastic_integrals}.
From our main results, we deduce: \cref{proposition:stable_fourth_moment}, a stable version of the recently proved fourth moment theorem on the Poisson space of~\cite{DoeblerPeccatiFMT,DoeblerVidottoZheng}; \cref{theorem:order_2_gaussian,theorem:order_2_poisson}, giving analytical criteria for conditionally normal or Poisson limit for order $2$ Itô-Wiener stochastic integrals.
\cref{section:application} contains the application to quadratic functionals of rescaled Poisson processes on the line.
In \cref{section:quantitative_results}, we show that, when the limit is a Gaussian mixture, we can adapt our strategy to establish a quantitative bound, \cref{section:improvement_1d}, we refine our results when the $F_{n}$'s are univariate, and we establish, in \cref{theorem:stable_convergence_wasserstein}, a bound in the Wasserstein transport distance.
We end the paper with some open questions.

\section*{Acknowledgements}
I am indebted to \emph{Christian Döbler} and \emph{Giovanni Peccati} for a careful reading of a preparatory version of this manuscript, and numerous helpful discussions and insightful comments during all the preparatory phase of this work.
I am also thankful to \emph{Matthias Reitzner} for an invitation to Osnabrück, and encouragements to study the case of a Poisson mixture.
I gratefully acknowledge support by the European Union through the European Research Council Advanced Grant for \emph{Karl-Theodor Sturm} \enquote{Metric measure spaces and Ricci curvature – analytic, geometric and probabilistic challenges}.
\tableofcontents%
\section{Preliminaries}%
\label{section:recollection}

\subsection{Notations}

In all this paper, the random variables are defined on a sufficiently big probability space $(\Omega, \gO, \dP)$.
We also fix a measurable space $(Z, \gZ)$ equipped with a $\sigma$-measure $\nu$.
For $q \in \dN$, we write $\nu^{q}$ for the $q$-fold product measure of $\nu$, and, for $p \in [0,\infty]$, we write $\CL^{p}(\nu) = \CL^{p}(Z, \gZ, \nu)$ for the Lebesgue space of $p$-integrable (equivalence classes) of functions.

\subsection{Probabilistic approximations and limit theorems}%
\label{subsection:approximations}

\subsubsection*{Stable convergence}
(See~\cite[VIII.5c]{JacodShiryaev}.)
Let $\gW$ be a sub-$\sigma$-algebra of $\gO$.
A sequence of $\gW$-measurable random variables $(F_{n})$ is said to \emph{converge stably} to a $\gO$-measurable random variable $F_{\infty}$ whenever, for all $Z \in \CL^{\infty}(\gW)$:
\begin{equation*}
(F_{n}, Z) \tto{law} (F_{\infty},Z).
\end{equation*}
This convergence is denoted by
\begin{equation*}
F_{n} \tto{stably} F_{\infty}.
\end{equation*}
Of course, stable convergence implies convergence in law but the reverse implication does not hold.
In practice, we use the following characterisation of stable convergence.
\begin{proposition}\label{proposition:stable_convergence_G}
  Let $(F_{n})$ be a sequence of $\gW$-measurable random variables, and $F_{\infty}$ be $\gO$-measurable.
  Let $\CI \subset \CL^{1}(\gW)$ be a linear space, and $\CG \subset \CL^{\infty}(\gW)$.
  Assume that $\sigma(\CI) = \sigma(\CG) = \gW$.
  The following are equivalent:
  \begin{enumerate}[(i)]
    \item $F_{n} \tto{stably} F_{\infty}$;
    \item for all $\phi$ continuous and bounded, and all $Z \in \CL^{\infty}(\gW)$: $\dE Z \phi(F_{n}) \tto{} \dE Z \phi(F_{\infty})$;\label{condition:stable_convergence_conditional}
    \item for all $G \in \CG$ and for all $\lambda \in \dR^{d}$: $\dE \e^{i \langle \lambda, F_{n} \rangle} G \tto{} \dE \e^{\langle \lambda, F_{\infty} \rangle} G$;\label{condition:stable_convergence_G}
    \item for all $I \in \CI^{d}$ and for all $\lambda \in \dR^{d}$: $\dE \e^{i \langle\lambda,F_{n} + I\rangle} \tto{} \dE \e^{i \langle \lambda, F_{\infty} + I\rangle}$.\label{condition:stable_convergence_I}
    \end{enumerate}
\end{proposition}
\begin{proof}
  Stable convergence is equivalent to \cref{condition:stable_convergence_conditional} by \cite[Proposition VIII.5.33.v]{JacodShiryaev}.
  Thus, stable convergence is also equivalent with \cref{condition:stable_convergence_G} since $\CG$ generates $\gW$.
  By linearity of $\CI$, \cref{condition:stable_convergence_I} implies that for all $J \in \CI$, all $t \in \dR$, and all $\lambda \in \dR^{d}$, as $n \to \infty$: $\dE \e^{it J} \e^{i \langle \lambda, F_{n} \rangle} \to \dE \e^{it J} \e^{i \langle \lambda, F_{\infty} \rangle}$.
  Letting $t \to 0$ in $(1-\e^{itJ})t^{-1} \to i J$, shows that $\dE J \e^{i \langle \lambda, F_{n} \rangle} \to \dE J \e^{i \langle \lambda, F_{\infty} \rangle}$, when $n \to \infty$.
  Since $\CI$ generates $\gW$, we conclude that \cref{condition:stable_convergence_I} implies stable convergence.
  The converse implication is immediate.
\end{proof}

\subsubsection*{Probabilistic variational distances}
The \emph{Wasserstein distance} between two $\dR^{d}$ random variables $X$ and $Y$ is defined by
\begin{equation*}
  d_{1}(X,Y) = \inf \dE |\tilde{X} - \tilde{Y}|,
\end{equation*}
where $|\cdot|$ is the Euclidean norm, and the infimum runs over all couple of random variables $(\tilde{X}, \tilde{Y})$ such that $\tilde{X}$ has the same law as $X$ and $\tilde{Y}$ has the same law as $Y$.
Due to the Kantorovich duality, the Wasserstein distance (see~\cite[Theorem 2.1]{GozlanLeonard}) between the laws of two integrable $\dR^{d}$-valued random variables $X$ and $Y$ can be rewritten:
\begin{equation*}
  d_{1}(X,Y) = \sup \dE \phi(X) - \dE \phi(Y),
\end{equation*}
where the supremum runs over all function $\phi \colon \dR^{d} \to \dR$ with Lipschitz constant not greater than $1$.
In this paper, as it is common when working with Stein's method, we consider a distance, whose variational formulation for two integrable $\dR^{d}$-valued random variables $X$ and $Y$ is given by
\begin{equation*}
  d_{3}(X,Y) = \sup \left\{ \dE \phi(X) - \dE \phi(Y) : \phi \in \CF_{3} \right\},
\end{equation*}
where $\CF_{3}$ if the set of all $\phi \colon \dR^{d} \to \dR$, thrice continuously differentiable with the second and third derivatives bounded by $1$.

\subsubsection*{Link with the convergence in law}
These two distances depend on $X$ and $Y$ only through their laws.
If $Y \sim \nu$, we sometimes write $d_{i}(X, \nu)$ for $d_{i}(X,Y)$ ($i \in \{1,3\}$).
The Wasserstein distance induces a topology on the space of probability measures that corresponds to the convergence in law together with the convergence of the first moment~\cite[Theorem 6.9]{VillaniOldAndNew}.
The distance $d_{3}$ induces a topology on the space of probability measures which is strictly stronger than the topology of the convergence in law.

\subsection{Definition of Poisson point processes}
We define $\CM_{\bar{\dN}}(Z)$  to be the space of all countable sums of $\dN$-valued measures on $(Z, \gZ)$.
The space $\CM_{\bar{\dN}}(Z)$ is endowed with the $\sigma$-algebra $\gM_{\bar{\dN}}(Z)$, generated by the \emph{cylindrical mappings}
\begin{equation*}
  \xi \in \CM_{\bar{\dN}}(Z) \mapsto \xi(B) \in \dN \cup \{\infty\},\quad B \in \gZ.
\end{equation*}
A random variable $\eta$ with values in $\CM_{\bar{\dN}}(Z)$ is a \emph{Poisson point process} (or \emph{Poisson random measure}) with intensity $\nu$ if the following two properties are satisfied:
\begin{enumerate}
  \item for all $B_{1}, \dots, B_{n} \in \gZ$ pairwise disjoint, $\eta(B_{1}), \dots, \eta(B_{n})$ are independent;
  \item for $B \in \gZ$ with $\nu(B) < \infty$, $\eta(B)$ is a Poisson random variable with mean $\nu(B)$.
\end{enumerate}
Poisson processes with $\sigma$-finite intensity exist \cite[Theorem 3.6]{LastPenrose}.
Moreover, by \cite[Corollary 3.7]{LastPenrose}, there exists a family of random variables $(X_{i})$ such that
\begin{equation*}
  law(\eta) = law\left(\sum_{i=0}^{\eta(Z)} \delta_{X_{i}}\right).
\end{equation*}
Since this paper is concerned only with distributional properties of Poisson point processes, we always assume that
\begin{equation*}
  \eta = \sum_{i=0}^{\eta(Z)} \delta_{X_{i}}.
\end{equation*}
We let $\gW$ be the $\sigma$-algebra generated by $\eta$.
Our definition of $\eta$ implies that $\gW \subset \gO$, and we often tacitly assume that $(\Omega, \gO, \dP)$ also supports random objects (such as a Brownian motion) independent of $\eta$.
We always look at stable convergence with respect to $\gW$.
However, for simplicity, unless otherwise specified, we assume that random variables are $\gW$-measurable.
In particular, we write $\CL^{2}(\dP)$ for $\CL^{2}(\Omega,\gW, \dP)$.

\subsection{Gaussian and Poisson mixtures}
As anticipated, we shall be interested in the stable convergence (with respect to $\gW$) of a sequence of Poisson functionals $(F_{n})$ to conditionally Gaussian and Poisson random variables.
Informally, we refer to such objects as \emph{Gaussian mixture} and \emph{Poisson mixture}.
Let $N$ be a standard Gaussian vector independent of $\eta$ and $S \in \CL^{2}(\gW)$.
We denote by $\lN(0,S^{2})$ the law of the Gaussian mixture $SN$.
Similarly, for $N$ a Poisson process on $\dR_{+}$ (with intensity the Lebesgue measure) independent of $\eta$ and $M \in \CL^{2}(\gW)$ non-negative, we write $\lP \lo(M)$ for the law of the (compensated) Poisson mixture $N(1_{[0,M]}) - M$.
We have a characterisation of these two laws in term of their conditional Fourier transforms:
$F \sim \lN(0,S^{2})$ if and only if
\begin{equation}\label{equation:conditional_Fourier_Gaussian_mixture}
  \dE [\e^{i \lambda F} | \eta] = \exp\left(-S^{2} \frac{\lambda^{2}}{2}\right);
\end{equation}
while $F \sim \lP\lo(M)$ if and only if 
\begin{equation}\label{equation:conditional_Fourier_Poisson_mixture}
  \dE [\e^{i \lambda F} | \eta] = \exp\left(M(\e^{i\lambda} - i \lambda - 1)\right).
\end{equation}

\subsection{Stochastic analysis for Poisson point processes}\label{section:stochastic_analysis}

\subsubsection*{The Mecke formula}
According to~\cite[Theorem 4.1]{LastPenrose}, we have for all measurable $f \colon \CM_{\bar{\dN}}(Z) \times Z \to [0,\infty]$:
\begin{equation}\label{equation:Mecke}
  \dE \int f(\eta, z) \eta(\dd z) = \int \dE f(\eta + \delta_{z}, z) \nu(\dd z).
\end{equation}
If $f$ is replaced by a measurable function with value in $\dR$ the previous formula still holds provided both sides of the identity are finite when we replace $f$ by $|f|$.

\subsubsection*{The representative of a functional}
For every random variable $F$ measurable with respect to $\eta$ we can write $F = f(\eta)$, for some measurable $f \colon \CM_{\bar{\dN}}(Z) \to \dR$ uniquely defined $\dP \circ \eta^{-1}$-almost surely on $(\CM_{\bar{\dN}}(Z), \gM_{\bar{\dN}}(Z))$.
We call such $f$ a \emph{representative} of $F$.
In this section, $F$ denotes a random variable, measurable with respect to $\sigma(\eta)$, and $f$ denotes one of its representatives.

\subsubsection*{The add and drop operators}
Given $z \in Z$, we let
\begin{align*}
  & D^{+}_{z} F =  f(\eta + \delta_{z}) - f(\eta); \\
  & D^{-}_{z} F =  (f(\eta) - f(\eta - \delta_{z})) 1_{z \in \eta}.
\end{align*}
The operator $D^{+}$ (resp.\ $D^{-}$) is called the \emph{add operator} (resp.\ \emph{drop operator}).
Due to the Mecke formula \cref{equation:Mecke}, these operations are well-defined on random variables (that is, $D^{+}$ and $D^{-}$ do not depend on the choice of the representative of $F$).
\begin{lemma}\label{lemma:D_bounded}
  Let $F \in \CL^{\infty}(\dP)$, then $D^{+}F \in \CL^{\infty}(\dP \otimes \nu)$.
\end{lemma}
\begin{proof}
  First of all, $\delta \colon Z \ni z \mapsto \delta_{z} \in \CM_{\bar{\dN}}(Z)$ is measurable (if $A$ is of the form $\{ \eta(B) = k \}$ for some $B \in \gZ$, then the pre-image by $\delta$ of $A$ is $B$, if $k = 1$; and the pre-image is empty, if $k > 1$).
  Hence, $D^{+}F$ is bi-measurable.
  Now let
  \begin{align}
    & U = \{t \in \dR,\, \text{such that}\ \dP(F \geq t) = 0 \}; \\
    & V = \{t \in \dR,\, \text{such that}\ (\dP \otimes \nu)(F + D_{z}^{+}F \geq t) = 0 \}.
  \end{align}
  By assumption $U \ne \emptyset$, and we want to show that $V \ne \emptyset$.
  Take $t \in U$, by the Mecke formula \cref{equation:Mecke}, we have that
  \begin{equation*}
    \dE \int 1_{\{F + D_{z}^{+}F \geq t\}} \nu(\dd z) = \dE \int 1_{\{F \geq t\}} \eta(\dd z) = 0.
  \end{equation*}
  Hence $t \in V$, this concludes the proof.
\end{proof}

\subsubsection*{Malliavin derivative}
For a random variable $F$, we write $F \in \dom D$ whenever: $F \in \CL^{2}(\dP)$ and
\begin{equation*}
  {|F|}_{1} := \int_{Z} \dE {(D_{z}^{+}F)}^{2} \nu(\dd z) < \infty.
\end{equation*}
Given $F \in \dom D$, we write $DF$ to denote the random mapping $DF \colon Z \ni z \mapsto D_{z}^{+}F$.
We regard $D$ as an unbounded operator $\CL^{2}(\dP) \to \CL^{2}(\dP \otimes \nu)$ with domain $\dom D$.

\subsubsection*{The divergence operator}
We consider the \emph{divergence operator} $\delta = D^{*} \colon \CL^{2}(\dP \otimes \nu) \to \CL^{2}(\nu)$, that is the unbounded adjoint of $D$.
Its domain $\dom \delta$ is composed of random functions $u \in \CL^{2}(\dP \otimes \nu)$ such that there exists a constant $c > 0$ such that
\begin{equation*}
  \left|\dE \int D^{+}_{z}F u(z) \nu(\dd z)\right| \leq c \sqrt{\dE F^{2}},\quad \forall F \in \dom D.
\end{equation*}
For $u \in \dom \delta$, the quantity $\delta u \in \CL^{2}(\dP)$ is completely characterised by the duality relation
\begin{equation}\label{equation:integration_by_parts_delta}
  \dE  G \delta u = \dE \int u(z) D_{z} F \nu(\dd z), \quad \forall F \in \dom D.
\end{equation}
If $h \in \CL^{2}(\nu)$, then $h \in \dom \delta$ and $\delta h = I_{1}(h)$.
From \cite[Theorem 5]{LastAnaSto}, we have the following Skorokhod isometry.
For $u \in \CL^{2}(\dP \otimes \nu)$, $u \in \dom \delta$ if and only if $\dE \int {(D_{z}^{+}u(z'))}^{2} \nu(\dd z) \nu(\dd z') < \infty$ and, in that case:
\begin{equation}\label{equation:Skorokhod_isometry}
  \dE {(\delta u)}^{2} = \dE \int {u(z)}^{2} \nu(\dd z) + \dE \int D_{z}^{+} u(z') D_{z'}^{+} u(z) \nu(\dd z) \nu(\dd z').
\end{equation}
The Skorokhod isometry implies the following Heisenberg commutation relation.
For all $u \in \dom \delta$, and all $z \in Z$ such that $z' \mapsto D_{z}^{+}u(z') \in \dom \delta$:
\begin{equation*}
  D_{z}\delta u = u(z) + \delta D_{z}^{+} u.
\end{equation*}
From \cite[Theorem 6]{LastAnaSto}, we have the following pathwise representation of the divergence: if $u \in \dom \delta \cap \CL^{1}(\dP \otimes \nu)$, then
\begin{equation}\label{equation:divergence_pathwise}
  \delta u = \int (1- D_{z}^{-}) u(z) \eta(\dd z) - \int u(z) \nu(\dd z).
\end{equation}
Note that $\dom \delta \cap \CL^{1}(\dP \otimes \nu)$ is dense in $\dom \delta$.

\subsubsection*{The Ornstein-Uhlenbeck generator}
The \emph{Ornstein-Uhlenbeck generator} $L$ is the unbounded self-adjoint operator on $\CL^{2}(\dP)$ verifying
\begin{equation*}
  \dom L = \{ F \in \dom D,\, \text{such that}\ DF \in \dom \delta \} \quad \text{and} \quad L = - \delta D.
\end{equation*}
Classically, $\dom L$ is endowed with the Hilbert norm $\dE F^{2} + \dE {(LF)}^{2}$.
The eigenvalues of $L$ are the non-positive integers and for $q \in \dN$ the eigenvectors associated are the so-called \emph{Wiener-Itô multiple integrals of order $q$}.
The kernel of $L$ coincides with the set of constants and the \emph{pseudo-inverse of $L$} is defined on the quotient $\CL^{2}(\dP) \setminus \ker L$, that is the space of centered square integrable random variables.
For $F \in \CL^{2}(\dP)$ with $\dE F = 0$, we have $LL^{-1}F = F$.
Moreover, if $F \in \dom L$, we have $L^{-1}LF = F$.
As a consequence of \cref{equation:Skorokhod_isometry}, $\dom D^{2} = \dom L$.
In particular, if $F$ has a vanishing expectation, then $L^{-1}F \in \dom D^{2}$.

\subsubsection*{The energy bracket}
As anticipated, a key object to consider in our study is the quantity $\nu(u DF)$, which is just the scalar product in $\CL^{2}(\nu)$ of the two random functions $u$ and $DF \in \CL^{2}(\nu \otimes \dP)$.
However, it turns out that the mapping $(F, G) \mapsto \nu(DF DG)$ is \emph{not} the carré du champ associate with $L$ (see, \cite{BouleauHirsch} for definitions).
Consequently, limit theorems formulated using the scalar product are not well-adapted to obtain convergence of stochastic integrals: this crucial observation allows \cite{DoeblerPeccatiFMT} to derive a fourth moment theorem in full generality on the Poisson space.
Given two elements $u \in \CL^{2}(\nu \otimes \dP)$ and $v \in \CL^{2}(\nu \otimes \dP)$ (possibly vector valued), we define the \emph{energy bracket} of $u$ and $v$: it is the random matrix
\begin{equation*}
  {[u,v]}_{\Gamma} = \frac{1}{2} \int u(z) \otimes v(z) \nu(\dd z) + \frac{1}{2} \int (1-D_{z}^{-})u(z) \otimes (1-D_{z}^{-})v(z) \eta(\dd z).
\end{equation*}
In the paper, we also consider the related object:
\begin{equation*}
  {[u, v]}_{\nu} = \nu(u \otimes v) = \int u(z) \otimes v(z) \nu(\dd z).
\end{equation*}
If $u$ and $v$ are real-valued, then ${[u, v]}_{\nu}$ is simply the scalar product of $u$ and $v$ in $\CL^{2}(\nu)$.
  By the Cauchy-Schwarz inequality ${[u,v]}_{\nu} \in \CL^{1}(\dP)$, and by the Mecke formula:
  \begin{equation}\label{equation:expectation_energy_bracket}
    \dE {[u,v]}_{\Gamma} = \dE {[u,v]}_{\nu}.
  \end{equation}
Moreover, if $F$ and $G \in \dom D$, we write
\begin{equation*}
  \Gamma(F,G) = {[DF, DG]}_{\Gamma}.
\end{equation*}
In \cite{HerryCarreDuChamp}, we prove that $\Gamma$ is indeed the carré du champ associated with the operator $L$ on the Poisson; this identity is our main motivation for introducing the energy bracket.
We also prove in \cite{HerryCarreDuChamp} that
\begin{equation}\label{equation:derivation_energy}
  \dE {[D(FG), u ]}_{\Gamma} = \dE G{[DF, u]}_{\Gamma} + \dE F{[DG, u]}_{\Gamma}.
\end{equation}
We denote by ${[u,v]}_{\widetilde{\beta}}$ the symmetrization of the matrix ${[u,v]}_{\beta}$ ($\beta \in \{ \Gamma, \nu \}$).

\subsubsection*{Test functions}
We say that a measurable function $\psi \colon Z \to \dR_{+}$ such that $\nu(\psi > 0) < \infty$ is a \emph{test function}.
We let $\CG \subset \CL^{\infty}(\dP)$ be the linear span of the random variables of the form $\e^{-\eta(\psi)}$, where $\psi$ is a test function.
Observe that $\CG$ is a sub-algebra of $\CA$ and that $\dom D$ is stable by multiplication by elements of $\CG$.
In view of~\cite[Lemma 2.2]{LastPenroseFockSpacePoisson} and its proof, we have that
\begin{proposition}\label{proposition:density_G}
  The set $\CG$ is dense in $\CL^{2}(\dP)$ (and in fact in every $\CL^{p}(\dP)$, $1 \leq p < \infty$).
  Moreover, the $\sigma$-algebra generated by $\CG$ coincides with $\gW$.
\end{proposition}

\subsubsection*{Extended Malliavin operators}
As mentioned above, we assume that $\gO$ is bigger than $\gW$.
However, every $\gO$-measurable random variable $F$ can be written $F = f(\eta, \Xi)$, where $\Xi$ is an additional randomness independent of $\eta$.
We define for every such $F$ the quantity $D_{z}^{+}F = f(\eta + \delta_{z}, \Xi) - f(\eta, \Xi)$.
It is an (easy) exercise to check that we can accordingly modify all the operators and functional spaces defined above, and that their properties are left unchanged.
Remark that our definition implies that, if $F$ is independent of $\eta$, then $D^{+}F = 0$, and that, if $F = ab$ with $a$ independent of $\gW$ and $b$ measurable with respect to $\gW$, $D^{+}F = a D^{+}b$.

\section{Main abstract results}%
\label{section:abstract_results}

\subsection*{Outline}
\cref{theorem:convergence_stable_quali} gives sufficient conditions for the stable convergence of a sequence of Poisson functionals to a Gaussian mixture.
While \cref{theorem:convergence_stable_quali_P} gives sufficient conditions for the stable convergence of a sequence of Poisson functionals to a Poisson mixture.
In \cref{section:quantitative_results}, we derive quantitative bounds for the convergence to a Gaussian mixture only.
However, in the case of Gaussian mixture, obtaining quantitative estimates requires to control additional terms and is quite technical.
This is why we treat first the simple qualitative bound both for Gaussian and Poisson mixtures and present the quantitative bound at the end.
\cref{theorem:convergence_stable_quantitative} is the quantitative counterpart of \cref{theorem:convergence_stable_quali} and provides bounds on the distance $d_{3}$ between the distribution of a Poisson functional and that of a Gaussian mixture.
We are not able to obtain quantitative estimates for the convergence to a Poisson mixture.
\cref{theorem:stable_convergence_wasserstein} is an improvement of our bound from the $d_{3}$ distance to the $d_{1}$ distance, when $(F_{n})$ is a sequence of univariate random variables.
An extended comparison of those results with the existing literature is carried out in \cref{section:comparison}.
All the proofs are given in \cref{section:proofs}.

\subsection{Main qualitative results}%
\label{section:qualitative_results}

\subsubsection{Convergence to a Gaussian mixture}
Recall that we study asymptotic for (possibly multivariate) random variables of the form $F_{n} = \delta u_{n}$.
In this setting, let us state the multivariate equivalent of \cref{condition:R3_0}:
\begin{equation*}
  \dE \int {|u_{n}(z)|} {|D_{z}^{+} F_{n}|}^{2} \nu(\dd z) \tto{} 0. \tag{R$_{3}$} \label{condition:R3}
\end{equation*}
We also consider
\begin{equation*}
  \dE \int {|D_{z}^{+}F_{n}|}^{4} \nu(\dd z) \tto{} 0. \tag{R$_{4}$} \label{condition:R4}
\end{equation*}
Remark that provided $(u_{n})$ is bounded in $\CL^{2}\left(\dP \otimes \nu\right)$, by the Cauchy-Schwarz inequality \cref{condition:R4} implies \cref{condition:R3}.
Several works about normal approximation of Poisson functionals (for instance, \cite{PeccatiSoleTaqquUtzet,LachiezeReyPeccatiContractions,LachiezeReyPeccatiMarkedProcess,ReitznerSchulteCLTUStatistics}) also consider conditions such as \cref{condition:R3,condition:R4}.
The random variable $u_{n} = - DL^{-1}F_{n}$ is always a solution of the equation $\delta u_{n} = F_{n}$ (other choices are possible).
Following \cite[Theorem 3.1]{PeccatiSoleTaqquUtzet} or \cite[Theorem 4.1]{DoeblerPeccatiFMT}, let us consider
\begin{equation}\label{condition:PSTU}
  -\nu(DL^{-1}F_{n} DF_{n}) = {[u_{n}, DF_{n}]}_{\nu} \tto{\CL^{1}(\dP)} \sigma^{2};
\end{equation}
or
\begin{equation}\label{condition:DP}
  -\Gamma(L^{-1}F_{n}, F_{n}) = {[u_{n}, DF_{n}]}_{\Gamma} \tto{\CL^{1}(\dP)} \sigma^{2}.
\end{equation}
Then, we have that \cref{condition:R3} with either \cref{condition:PSTU} or \cref{condition:DP} imply that $F_{n} \tto{law} \lN(0,\sigma^{2})$.
In our setting of random variance it is thus very natural to consider one of the following conditions:
  \begin{equation*}
    {[u_{n}, D F_{n}]}_{\nu} \tto{\CL^{1}(\dP)} S S^{T} \tag{S$_{\nu}$}\label{condition:Snu};
  \end{equation*}
  or
\begin{equation*}
  {[u_{n}, D F_{n}]}_{\Gamma} \tto{\CL^{1}(\dP)} S S^{T} \tag{S$_{\Gamma}$}\label{condition:S};
  \end{equation*}
  for some $S \in \CL^{2}(\dP)$.
  When dealing with stable convergence, either of the following conditions would guarantee asymptotic independence:
\begin{equation*}
  {[u_{n}, h]}_{\nu} \tto{\CL^{1}(\dP)} 0, \quad \forall h \in \CL^{2}(\nu) \tag{W$_{\nu}$}\label{condition:Wnu};
\end{equation*}
or
\begin{equation*}
  {[u_{n}, DG]}_{\Gamma} \tto{\CL^{1}(\dP)} 0, \quad \forall G \in \CG \tag{W$_{\Gamma}$}\label{condition:W}.
\end{equation*}
Our first statement regarding stable limit theorems on the Poisson space is the following qualitative generalization of the results of \cite{PeccatiSoleTaqquUtzet,DoeblerPeccatiFMT} to consider Gaussian mixtures in the limit.
\begin{theorem}\label{theorem:convergence_stable_quali}
  Let $\{F_{n} = (F_{n}^{(1)}, \dots, F_{n}^{(d)});\; n \in \dN\} \subset \dom D$.
  Assume that, for all $n \in \dN$, there exists $u_{n} \in \dom \delta$ such that $F_{n} = \delta u_{n}$ and that \cref{condition:R3} holds.
  Let $S = (S_{1}, \dots, S_{d}) \in \CL^{2}(\dP)$.
  Assume that either \cref{condition:Wnu,condition:Snu} holds; or \cref{condition:W,condition:S} holds.
  Then $F_{n} \tto{stably} \lN(0,S^{2})$.
\end{theorem}

\begin{remark}
  The condition \cref{condition:S} is a priori more involved than \cref{condition:Snu}: indeed integrating with respect to $\eta$ adds some randomness to the object.
  However, in \cref{section:fourth_moment} we need the result involving ${[\cdot,\cdot]}_{\Gamma}$ in order to obtain a stable version of the fourth moment theorem of \cite{DoeblerPeccatiFMT}.
    On the other hand, we use conditions of type \cref{condition:Snu,condition:Wnu} to derive \cref{theorem:order_2_gaussian,theorem:order_2_poisson}.
\end{remark}

\subsubsection{Convergence to a Poisson mixture}
Here we only consider univariate random variables.
Convergence in law of Poisson functionals to a Poisson distribution represents another archetypal limit theorem.
In the setting of the Malliavin-Stein method, \cite{PeccatiChenStein} proves that the two conditions:
\begin{equation*}
  -\nu(D^{+}L^{-1}F_{n} D^{+}F_{n}) \tto{} m,
\end{equation*}
and
\begin{equation*}
  \dE \int |D_{z}^{+}L^{-1}F_{n} D_{z}^{+}F_{n}(D_{z}^{+}F_{n}-1)| \nu(\dd z) \tto{} 0,
\end{equation*}
imply that $F_{n} \tto{law} \lP\lo(m)$\footnote{\cite{PeccatiChenStein} works with non-centered random variables but the result is equivalent.}.
It is thus very natural to replace \cref{condition:R3} by the following asymptotic conditions for $F_{n} = \delta u_{n}$ (here we only considered scalar-valued random variables):
\begin{equation*}
  \dE \int |u_{n}(z) D_{z}^{+}F_{n} (D_{z}^{+} F_{n} - 1)| \nu(\dd z) \tto{} 0. \tag{P$_{3}$} \label{condition:P3}
\end{equation*}
We also consider the Poisson version of \cref{condition:R4}:
\begin{equation*}
  \dE \int {|D_{z}^{+}F_{n}|}^{2} {|D_{z}^{+}F_{n} - 1|}^{2} \nu(\dd z) \tto{} 0. \tag{P$_{4}$} \label{condition:P4}
\end{equation*}
Again, provided $(u_{n})$ is bounded in $\CL^{2}(\nu \otimes \dP)$, we see that \cref{condition:P4} implies \cref{condition:P3}.
With this notation, we have the following qualitative result for convergence to a Poisson mixture.
\begin{theorem}\label{theorem:convergence_stable_quali_P}
  Let $(F_{n}) \subset \dom D$.
  Let $M \in \CL^{1}(\dP)$ with $M \geq 0$.
  Assume that, for all $n \in \dN$, there exists $u_{n} \in \dom \delta$ such that $F_{n} = \delta u_{n}$ and that \cref{condition:P3,condition:Wnu} hold, and moreover assume that
  \begin{equation*}
    {[u_{n}, DF_{n}]}_{\nu} = \langle u_{n}, D F_{n} \rangle_{\CL^{2}(\nu)} \tto{\CL^{1}(\dP)} M \tag{M$_{\nu}$}\label{condition:M}.
  \end{equation*}
  Then $F_{n} \tto{stably} \lP\lo(M)$.
\end{theorem}

\begin{remark}\label{remark:MS}
  \cref{condition:M} is formally equivalent to \cref{condition:Snu} (we can always write $S^{2} = M$).
  However, it is important to note that our theorem cannot be true if we replace the scalar product by the energy bracket in \cref{condition:M}, that is that we work with the condition:
  \begin{equation*}
    {[u_{n}, DF_{n}]}_{\Gamma} \tto{\CL^{1}(\dP)} M \tag{M$_{\Gamma}$}\label{condition:M_gamma}.
  \end{equation*}
  Indeed take $F = \eta(A) - \nu(A)$, with $A \in \gZ$, $\nu(A) <\infty$.
  We can write $F = \delta 1_{A}$, and $DF = 1_{A}$, hence \cref{condition:P3} is satisfies, since we have
  \begin{equation*}
    \int_{A} |1_{A} - 1| \dd \nu = 0.
  \end{equation*}
  On the other hand, we have that
  \begin{equation*}
    {[1_{A}, 1_{A}]}_{\Gamma} = \frac{1}{2}(\nu(A) + \eta(A)) = M.
  \end{equation*}
  Let $F \sim \lP\lo(M)$, then
  \begin{equation*}
  \begin{split}
    \dE \e^{i\lambda F} &= \dE \exp(M(\e^{i\lambda} - i\lambda -1)) \\
                        &= \exp\left(\frac{1}{2}\nu(A)(\e^{i\lambda} - i\lambda - 1)\right) \exp\left(\frac{\nu(A)}{2}\left( \exp(\e^{i \lambda} -i \lambda - 1) - 1\right)\right).
  \end{split}
  \end{equation*}
  Hence, we see that the law of $F$ is not the one of $\eta(A) - \nu(A)$.
  Remark that \cref{condition:S,condition:M_gamma} are also formally equivalent.
  At a more structural level, \cite{DoeblerPeccatiProduct} proves that if a sequence $(F_{n})$ of Poisson stochastic integrals satisfies a deterministic reinforcement of \cref{condition:S}, that is:
  \begin{equation*}
    {[-DL^{-1}F_{n}, DF_{n}]}_{\Gamma} = -\Gamma(L^{-1}F_{n}, F_{n}) \tto{\CL^{2}(\dP)} \sigma^{2},
  \end{equation*}
   then, without further assumptions, the sequence converges in law to a Gaussian.
   Since the condition of \cite{DoeblerPeccatiProduct} implies \cref{condition:M_gamma} for the particular choice $u_{n} = - DL^{-1}F_{n}$ when $F_{n}$ is a stochastic integral, we see that \cref{condition:M_gamma} could not enforce convergence to a Poisson mixture.
\end{remark}

\subsection{Main quantitative results for Gaussian mixtures}\label{section:quantitative_results}

\subsubsection{General results in any dimension}
In this section, we obtain quantitative Malliavin-Stein bounds between the law of a Poisson functional and that of a Gaussian mixture.
As for \cref{theorem:convergence_stable_quali} we can either work with ${[\cdot,\cdot]}_{\nu}$ or with ${[\cdot,\cdot]}_{\Gamma}$ yielding to different bounds.
Results involving ${[\cdot,\cdot]}_{\nu}$ are a priori easier to handle in applications.
However, we state the two bounds for completeness.
For short, for $\phi \in \CC^{k}(\dR^{d})$, let us write $\Phi_{k} =  \sup_{x \in \dR^{d}} |\nabla^{k} \phi|(x)$, and $S \in \cov$ whenever $S \in \dom D$ with $S S^{T} \in \dom D$.
Also recall that we write ${[\cdot,\cdot]}_{\widetilde{\beta}}$ for the symmetrization of the random matrix ${[\cdot,\cdot]}_{\beta}$, $\beta \in \{\nu, \Gamma\}$, defined in \cref{section:stochastic_analysis}.
We are now in position to state our bound in the $d_{3}$ distance of a Poisson functional to a Gaussian mixture.
\begin{theorem}\label{theorem:convergence_stable_quantitative}
  Let $\beta \in \{\nu, \Gamma\}$.
  Let $F \in \dom D$, and $S \in \cov$.
  Then,
  \begin{equation*}
    \begin{split}
      d_{3}(F, \lN(0,S^{2})) & \leq \frac{1}{4} \dE {\left| {{[u,DF]}}_{\widetilde{\beta}} - S S^{T}\right|} \\
                             &+ \frac{1}{3} \dE {\left|{[u, (DS)S^{T}]}_{\beta}\right|} \\
                             &+ \frac{1}{6} \dE \int {|u(z)|}( {|D_{z}F|}^{2} + {|DS|}^{2}) \nu(\dd z).
    \end{split}
  \end{equation*}
\end{theorem}
In \cref{theorem:convergence_stable_quali}, \cref{condition:Snu} enforces that the asymptotic covariance $S$ is measurable with respect to $\eta$.
Thanks to \cref{proposition:quantitative_bound_energy_D}, when $S_{n}^{2} = {{[u_{n}, DF_{n}]}}_{\widetilde{\nu}}$ is non-negative, we can deduce sufficient conditions for the stable convergence of a Poisson functional that involves stable convergence of $S_{n}$ to some $S$ (not necessarily measurable with respect to $\eta$).
This weaker form of convergence can allow, for instance, $S$ to be independent of $\eta$.
\begin{theorem}\label{theorem:convergence_stable}
  Let ${(F_{n})}_{n \in \dN} \subset \dom D$, and $S \in \CL^{2}(\Omega)$ (not necessarily measurable with respect to $\eta$).
  Let ${(u_{n})}_{n \in \dN} \subset \dom \delta$ such that $F_{n} = \delta u_{n}$ for all $n \in \dN$, and \cref{condition:Wnu,condition:R3} holds.
    Assume, moreover, that for $n$ sufficiently big ${{[u_{n}, DF_{n}]}}_{\widetilde{\nu}} = C_{n} + \epsilon_{n}$, where $C_{n} = S_{n} S_{n}^{T}$ is a symmetric non-negative random matrix, and:
  \begin{align}
    & C_{n} \tto{stably} S S^{T};\label{condition:Cst}\tag{C.st} \\
    & \epsilon_{n} \tto{\CL^{1}(\dP)} 0;\label{condition:epsilon}\tag{$\epsilon$} \\
    & {[u_{n}, (DS_{n})S_{n}]}_{\nu} \tto{\CL^{1}(\dP)} 0;\label{condition:RS}\tag{RS} \\
    & \int {|h(z)|} {|D_{z}^{+}F_{n}|}^{2} \nu(\dd z) \tto{\CL^{1}(\dP)} 0; \label{condition:Rh} \tag{Rh} \\
    & \int {|u_{n}(z)|} {|D_{z}^{+}S_{n}|}^{2} \nu(\dd z) \tto{\CL^{1}(\dP)} 0; \label{condition:S3} \tag{S$_{3}$} \\
    & \int {|h(z)|} {|D_{z}^{+}S_{n}|}^{2} \nu(\dd z) \tto{\CL^{1}(\dP)} 0. \label{condition:Sh} \tag{Sh}
  \end{align}
  Then $F_{n} \tto{stably} \lN(0,S^{2})$.
\end{theorem}
\begin{remark}
  We can also consider 
  \begin{equation*}
    \int {|D_{z}^{+}S_{n}|}^{4} \nu(\dd z) \tto{\CL^{1}(\dP)} 0. \label{condition:S4} \tag{S$_{4}$}
  \end{equation*}
  Provided $(u_{n})$ is bounded in $\CL^{2}(\dP \otimes \nu)$ then \cref{condition:R4} implies \cref{condition:R3,condition:Rh}, and \cref{condition:S4} implies \cref{condition:S3,condition:Sh}.
\end{remark}

\begin{remark}
  We formulated our result with ${[\cdot,\cdot]}_{\nu}$; we could do the same for ${[\cdot,\cdot]}_{\Gamma}$.
  Details are left to the reader.
\end{remark}

\subsubsection{Bounds in the Wasserstein distance for the one-dimensional case}%
\label{section:improvement_1d}
The results of the previous section are stated in the rather abstract distance $d_{3}$.
When $F$ is univariate, one can use a regularization lemma in order to turn the estimates for the $d_{3}$ into estimates for the Wasserstein distance $d_{1}$.
In this section, all the random variables are implicitly univariate.
\begin{theorem}\label{theorem:stable_convergence_wasserstein}
  Let $F \in \dom D$ such that $F = \delta u$ for some $u \in \dom \delta$, and let $S \in \cov$.
  Consider
  \begin{align*}
    & \Delta_{1} = {\left(\frac{2}{\pi}\right)}^{\frac{1}{2}} (2 + \dE |S|) + \dE |F|; \\
    & \Delta_{2} = \frac{1}{4} {\left(\frac{2}{\pi}\right)}^{\frac{1}{2}}  \dE | \nu(u DF) - S^{2}| + 2^{\frac{1}{2}} \left(\frac{1}{3} \dE |S \nu(u DS)| + \frac{1}{6} \dE \nu\left(|u|\left({|DF|}^{2} + {|DS|}^{2}\right)\right)\right).
  \end{align*}
  Then, we have that
  \begin{equation*}
    d_{1}(F, \lN(0,S^{2})) \leq  \max\left(\left(2^{\frac{1}{3}} + 2^{-\frac{2}{3}}\right) \Delta_{1}^{\frac{2}{3}} \Delta_{2}^{\frac{1}{3}}, \Delta_{1} \Delta_{2}\right).
  \end{equation*}
\end{theorem}
This theorem allows us to prove a quantitative version of \cref{theorem:convergence_stable} in the univariate case.
\begin{theorem}\label{theorem:stable_convergence_wasserstein_sequence}
  Let the assumptions and the notations of \cref{theorem:convergence_stable} prevail.
  Consider
  \begin{align*}
    & \Delta_{1,n} = {\left(\frac{2}{\pi}\right)}^{\frac{1}{2}} (2 + \dE |S_{n}|) + \dE |F_{n}|; \\
    & \Delta_{2,n} = \frac{1}{4} {\left(\frac{2}{\pi}\right)}^{\frac{1}{2}}  \dE |\epsilon_{n}| + 2^{\frac{1}{2}} \left(\frac{1}{3} \dE |S_{n} \nu(u_{n} DS_{n})| + \frac{1}{6} \dE \nu\left(|u_{n}|\left({|DF_{n}|}^{2} + {|DS_{n}|}^{2}\right)\right)\right).
  \end{align*}
  Then,
  \begin{equation*}
    d_{1}(F_{n}, \lN(0,S^{2})) \leq \max\left(\left(2^{\frac{1}{3}} + 2^{-\frac{2}{3}}\right) \Delta_{1,n}^{\frac{2}{3}} \Delta_{2,n}^{\frac{1}{3}}, \Delta_{1,n} \Delta_{2,n}\right) + {\left(\frac{2}{\pi}\right)}^{\frac{1}{2}} d_{1}(S_{n}, S).
  \end{equation*}
\end{theorem}

\subsection{Comparison with existing results}%
\label{section:comparison}
First, on the Gaussian space, the authors of~\cite{NourdinNualartLimitsSkorokhod,HarnettNualartCLTStratanovich,NourdinNualartPeccatiStableLimits} work with iterated Skorokhod integrals of any order $q \in \dN$.
That is, given a Gaussian functional $F$ and given $u$ such that $F = \delta^{q}u$, they give probabilistic conditions in terms of $u$ and $F$ for stable convergence of $F$ to a Gaussian mixture.
\cref{theorem:convergence_stable_quali,theorem:convergence_stable_quantitative} are the Poisson version of their results for the case $q=1$.
Due to the lack of diffusiveness on the Poisson space, it does not seem possible to reach a result involving iterated Kabanov integrals, via our method of proof, that is, via integration by parts.

Second, \cref{condition:S} enforces that the convergence of $C_{\Gamma} = [u, DF]$ (or its symmetrized version) determines the asymptotic covariance.
The comparison of $C_{\Gamma}$ and $S S^{T}$ is similar in the Gaussian case~\cite{NourdinNualartLimitsSkorokhod}: the quantity $\langle DF, u \rangle$ (where $D$ is the Malliavin derivative on the Gaussian space) controls the asymptotic variance of the functional $F = \delta u$.
In this respect, let us refer to~\cite[Theorem 5.3.1]{NourdinPeccatiBlueBook} for deterministic variance (for the choice $u = -DL^{-1}F$), to~\cite[Theorem 3.1]{NourdinNualartLimitsSkorokhod}, to~\cite[Theorem 3.2]{HarnettNualartCLTStratanovich} and to~\cite[Theorem 5.1]{NourdinNualartPeccatiStableLimits} for random asymptotic variances.
However, we see from \cref{condition:Snu} that another relevant quantity to consider is $C_{\nu} = {[DF_{n}, u_{n}]}_{\nu}$.
The matrix $C_{\nu}$ would also correspond in the Gaussian setting to $\langle u, DF \rangle$ since $\Gamma(F)$ and ${|DF|}^{2}$ coincide on the Gaussian space.
As already observed by~\cite{DoeblerPeccatiFMT}, working with $C_{\Gamma}$ rather than $C_{\nu}$ is critical in obtaining a fourth moment theorem.
We also work with $C_{\Gamma}$ to obtain our stable version of their fourth moment theorem.
When working with deterministic covariances one can choose $C_{\nu}$ and still obtain sufficient conditions for convergence of Poisson functionals to a Gaussian (see, for instance~\cite{LachiezeReyPeccatiContractions,LachiezeReyPeccatiMarkedProcess,ReitznerSchulteCLTUStatistics}).

Our condition \cref{condition:Wnu} is the exact counterpart of the condition $\langle u_{n}, h \rangle \to 0$ (see~\cite[Remark 3.2]{NourdinNualartLimitsSkorokhod}) in the Gaussian setting, enforcing some asymptotic independence.
When working with the energy bracket, we have \cref{condition:W} that we can also regard as an asymptotic independence condition.
\cref{condition:RS} plays the same role, in our setting, as $\langle u, DS^{2} \rangle \to 0$ in~\cite{NourdinNualartPeccatiStableLimits}.
On the Gaussian space, by the chain rule, $DS^{2} = 2S DS$.
In our case we cannot have this simplification, which implies that we have to formulate our condition in terms of $SDS$.
This adds an extra difficulty since, in practice, the convergence of $C_{\nu}$ or $C_{\Gamma}$ only provides information on $S S^{T}$ but not on $S$.
As the condition with $DS^{2}$ is already present in the Gaussian setting \cite{NourdinNualartPeccatiStableLimits}, we do not expect that the condition \cref{condition:RS} could disappear in general.
The condition \cref{condition:R3} is specific to the Poisson setting.
Controlling quantities of the form $\int |D_{z}^{+}L^{-1} F| {|D_{z}^{+}F|}^{2} \nu(\dd z)$ is standard in the theory of limit theorems for Poisson functionals and already appeared in the first result on the Malliavin-Stein method on the Poisson space~\cite[Theorem 3.1]{PeccatiSoleTaqquUtzet}, as well as in the proof of the fourth moment theorem on the Poisson space~\cite[Equation 4.2]{DoeblerPeccatiFMT}.
These correspond to the choice $u = -DL^{-1}F$ in \cref{condition:R3}.
In our case we have an extra term of the form $\int |u(z)| {|D_{z}^{+}S|}^{2} \nu(\dd z)$.
This term is also the result of the lack of a chain rule and we do not expect we could remove it.

Furthermore, the authors of~\cite{NourdinNualartLimitsSkorokhod,HarnettNualartCLTStratanovich,NourdinNualartPeccatiStableLimits} only consider results involving the convergence in $\CL^{1}(\dP)$ of the Stein matrix $C_{\nu}$, thus imposing measurability with respect to the underlying Gaussian process on the limit covariance.
In our case, when the limiting covariance is non-negative, we can replace the condition of convergence in $\CL^{1}(\dP)$ by the weaker form of stable convergence to obtain \cref{theorem:convergence_stable}.
This modification relies on our quantitative bounds, which is why, in this case we need to check \cref{condition:RS} while \cref{theorem:convergence_stable_quali} does not need to enforce this condition.
Being quantitative, the results of~\cite{NourdinNualartPeccatiStableLimits} could also be modified in order to obtain a result similar to \cref{theorem:convergence_stable} with the same proof as the one we gave in the Poisson setting.

Lastly, in the multidimensional case, our bound in \cref{theorem:convergence_stable_quantitative} holds for every symmetric covariance random matrix $C = SS^{T}$, while the results of \cite{NourdinNualartPeccatiStableLimits} are limited to the case of a diagonal matrix.
  \cite{HarnettNualartCLTStratanovich} also deals with generic matrices but relies on the so-called method of the characteristic function that is not known to provide quantitative bounds.

On the other hand, the convergence to Poisson mixtures was not considered for Gaussian functionals (recall that by \cite[Theorem 2.10.1]{NourdinPeccatiBlueBook} random variables in a fixed Wiener chaos are absolutely continuous with respect to the Lebesgue measure).
Several authors have applied the Malliavin-Stein approach on the Poisson space to consider convergence to a Poisson random variable with deterministic mean.
The work of \fcite{PeccatiChenStein} is the first result in that direction.
Selecting $u_{n} = - DL^{-1}F_{n}$ and $M = \dE M = c$ in \cref{condition:M} exactly yields the condition of \cite[Proposition 3.3]{PeccatiChenStein}: $\langle -DL^{-1}F_{n}, DF_{n} \rangle_{\CL^{2}(\nu)} \to c$ (remark that \cite{PeccatiChenStein} works with non-centered random variables).
For Poisson approximation, the above discussion on the difference between $S_{D}$ and $S_{\Gamma}$ does not apply as we only obtain a condition involving $S_{D}$ (see \cref{remark:MS}).
Our condition \cref{condition:P3} is similar to the one in \cite{PeccatiChenStein}.

Contrary to \cite{PeccatiChenStein}, we cannot obtain quantitative bounds for Poisson approximation.
In fact, we do not know how to adapt the methods in \cref{section:quantitative_results} to reach estimates for the distance of a Poisson functional to a Poisson mixture.
Indeed, our approach towards quantitative estimates relies on the computability of the Malliavin derivative of a Gaussian mixture, since they always can be written $SN$ with $N$ independent of $\eta$, and in this case $D(SN) = (DS)N$.
However, if $N(M)$ is a Poisson mixture directed by $M$, we have:
\begin{equation*}
  D_{z} N(M) = N(1_{[0,M+D_{z}M]}) - N(1_{[0,M]}) - D_{z}M.
\end{equation*}
The computations with this quantity seem not tractable, and we need new techniques to tackle this problem; we reserve exploring this direction of research for future works.

\subsection{Proofs}\label{section:proofs}
\subsubsection*{General strategy}
Since Stein equations for Gaussian or Poisson mixtures are not available, we use an interpolation method (employed by \cite{NourdinNualartPeccatiStableLimits} in the Gaussian setting) or a characteristic function method (used by \cite{NourdinNualartLimitsSkorokhod} in the Gaussian setting) that consists of obtaining a differential equation for the conditional Fourier transform.
As it is common in the Malliavin-Stein setting, regarding both methods, we obtain the convergence of $(F_{n})$ by controlling quantities of the form $\dE F_{n} \phi(F_{n})$, where $\phi$ varies within a class of smooth functions.
Since we assume that $F_{n} = \delta u_{n}$, our strategy exploits the duality between $\delta$ and $D$ to write $\dE F_{n} \phi(F_{n}) = \dE {[u_{n}, D\phi(F_{n})]}_{\beta}$ for $\beta \in \{ \nu, \Gamma\}$.
We, then, develop $D\phi(F)$ using a discrete equivalent of the chain rule at the level of the Poisson space.
We will use this strategy of integration by parts several times; the structure of the argument being the same every time, we first state and prove generic lemmas before proceeding to the proof of our main results.

\subsubsection{Pseudo chain rules and integration by parts}
\subsubsection*{Substitute for the chain rule}
The Markov generator $L$ is not a diffusion (see \cite[Equation 1.3]{LedouxGeometryMarkov}).
Likewise, the add operator $D^{+}$ and drop operator $D^{-}$ are not derivations (see \cite[Chapter III, Section 10]{BourbakiAlgebra} for details on derivations).
In particular, the classical chain rule does not apply, that is, for a generic smooth function $\phi \colon \dR^{d} \to \dR$ and a random variable $F$:
\begin{equation*}
  D\phi(F) \ne \langle \nabla \phi(F), DF \rangle.
\end{equation*}
However, since $D^{+}\phi(F) = \phi(F + D^{+}F) - \phi(F)$, writing, for $z \in Z$:
\begin{equation*}
  R_{z}^{+}(\phi, F) = D_{z}^{+}\phi(F) - \langle \nabla \phi(F), D_{z}^{+} F \rangle,
\end{equation*}
and applying the fundamental theorem of calculus we obtain that
\begin{equation}\label{equation:chain_rule_plus}
  |R_{z}^{+}(F, \phi)| \leq {|\nabla^{2} \phi|}_{\infty} {|D_{z}^{+}F|}^{2}, \qquad \forall z \in Z.
\end{equation}
A similar formula holds for $D_{z}^{-}$, that is with
\begin{equation*}
  R_{z}^{-}(F,\phi) = D_{z}^{-}\phi(F) - \langle \nabla \phi(F), D_{z}^{-}F \rangle,
\end{equation*}
we find that
\begin{equation}\label{equation:chain_rule_minus}
  |R_{z}^{-}(F, \phi)| \leq {|\nabla^{2} \phi|}_{\infty} {|D_{z}^{-}F|}^{2}, \qquad \forall z \in Z.
\end{equation}
Let us also observe that the definitions of $R^{+}$ and $R^{-}$ still make sense when $\phi$ is $\dR^{q}$-valued.
In this case, $\nabla \phi$ is the Jacobian matrix of $\phi$ and $\langle \nabla \phi(F), D^{\pm}F \rangle$ is replaced with the product $\nabla \phi(F) D^{\pm}F$.

\subsubsection*{Taylor formula for difference operators}
Another possible approach to substitute to the chain rule is to use finite differences that will be useful when targeting Poisson mixtures.
This yields the following quantity:
\begin{equation*}
  P_{z}(\phi, F) = D_{z}^{+}\phi(F) - D_{z}^{+}F(\phi(F+1) - \phi(F)),
\end{equation*}
where $\phi \colon \dR \to \dR$ is smooth and $F$ is a random variable.
Another application of Taylor's formula gives us the following discrete counterpart of the chain rule.
\begin{equation}\label{equation:chain_rule_difference}
  |P_{z}(F, \phi)| \leq  |D^{+}F (D^{+}F - 1)| {|\nabla^{2} \phi|}_{\infty}.
\end{equation}
note that 
\begin{remark}
  It is possible to obtain a similar formula for $D^{-}$ or on $\dR^{d}$ but we have no use for it.
\end{remark}

\subsubsection*{Integration by parts formulae}
Those integration by parts formulae at the level of the Poisson space are obtained with Malliavin calculus.
For short, let us also write $u^{-}(z) = (1-D_{z}^{-})u(z)$ whenever $u \in \CL^{2}(\nu \otimes \dP)$.
\begin{lemma}\label{lemma:integration_by_parts_energy}
  Let $F = (F_{1}, \dots, F_{q}) \in \dom D$, $u = (u_{1}, \dots, u_{d}) \in \dom \delta$, and $G \in \CG$.
  Let $\phi \colon \dR^{q} \to \dR^{d}$, twice continuously differentiable and with bounded derivatives.
  Assume that, for $l \in \{1,2\}$, $\int {|u(z)|} {|D_{z}^{+}F|}^{l} \nu(\dd z) < \infty$.
  Then:
  \begin{equation}\label{equation:integration_by_parts_energy}
    \begin{split}
      \dE \langle \phi(F)G, \delta u \rangle &= \dE G \left\langle \nabla \phi(F), {[u,DF]}_{\Gamma} \right\rangle \\
                                                    &+ \dE \left\langle \phi(F), {[u,DG]}_{\Gamma} \right\rangle \\
                                                    &+ \frac{1}{2} \dE G \int \left\langle u(z), R^{+}_{z}(\phi, F) \right\rangle \nu(\dd z) \\
                                                    &+  \frac{1}{2} \dE G \int \left\langle u^{-}(z), R^{-}_{z}(\phi, F) \right\rangle \eta(\dd z).
    \end{split}
  \end{equation}
\end{lemma}

\begin{proof}
  We write $A = \langle \phi(F), \delta u \rangle$, $B = G \left\langle \nabla \phi(F), {[u,DF]}_{\Gamma} \right\rangle$, $C =  \left\langle \phi(F), {[u,DG]}_{\Gamma} \right\rangle$, and
  \begin{align*}
    & K^{+} = \frac{1}{2} G \int \left\langle u(z), R^{+}_{z}(\phi, F) \right\rangle \nu(\dd z);\\
    & K^{-} = \frac{1}{2} G \int \left\langle u^{-}(z), R^{-}_{z}(\phi, F) \right\rangle \eta(\dd z).
  \end{align*}
    First, let us check that every term is well defined.
    Since the derivatives of $\phi$ are bounded, $\phi$ is Lipschitz.
    Since $F \in \dom D$, we find that $\phi(F) \in \dom D$ and $G \phi(F) \in \dom D$.
    Since $u \in \dom \delta$, we have that $\delta u \in \CL^{2}(\dP)$ and, then, $A \in \CL^{1}(\dP)$.
    Applying the Cauchy-Schwarz inequality and the Mecke formula, we find, in view of the assumptions and of \cref{equation:chain_rule_plus,equation:chain_rule_minus}:
    \begin{align*}
      & \dE |B| \leq {|\nabla \phi|}_{\infty} {|G|}_{\CL^{\infty}(\dP)} \dE \int {|u(z)|} {|D_{z}^{+} F|} < \infty; \\
      & \dE |C| \leq {|\phi|}_{\infty} \dE \int {|u(z)|} {|D_{z}^{+} G|} \nu(\dd z) < \infty; \\
      & \dE |K^{+}| + \dE |K^{-}| \leq {|\nabla^{2} \phi(F)|}_{\infty} {|G|}_{\CL^{\infty}(\dP)} \dE \int {|u(z)|} {|D_{z}^{+} F|}^{2} \nu(\dd z) < \infty.
    \end{align*}
    These estimates also justify the use of the Mecke formula on non-necessarily non-negative quantities that we do in the rest of the proof.
    Now, we prove the equality \cref{equation:integration_by_parts_energy}.
    Let $D = B + C + R^{+} - R^{-}$.
    By integration by parts \cref{equation:integration_by_parts_delta}, we find
    \begin{equation}\label{equation:A}
      \dE A = \dE \int \left\langle D_{z}^{+}(\phi(F)G), u(z) \right\rangle \nu(\dd z).
    \end{equation}
    By the Mecke formula \cref{equation:Mecke}, by \cref{equation:derivation_energy}, and by the fact that all the terms are integrable, we get
    \begin{equation*}
      \dE A = \dE {[D(\phi(F)G), u]}_{\Gamma} = \dE \phi(F) {[D G, u]}_{\Gamma} + \dE G {[D\phi(F), u ]}_{\Gamma},
    \end{equation*}
    We conclude the proves by definition of the energy bracket and of $R^{+}$ and $R^{-}$.
\end{proof}
When $G=1$, we can directly use the definition of $R^{+}$ in \cref{equation:A}, this yields the following integration by parts involving ${[\cdot,\cdot]}_{\nu}$ rather than ${[\cdot,\cdot]}_{\Gamma}$.
\begin{lemma}\label{lemma:integration_by_parts_D}
  Under the same assumptions as for \cref{lemma:integration_by_parts_energy}, it holds
  \begin{equation*}
    \begin{split}
      \dE \langle \phi(F), \delta u \rangle &= \dE \left\langle \nabla \phi(F), {[u, DF]}_{\nu} \right\rangle \\
                                                   &+ \dE \int \left\langle u(z), R_{z}^{+}(\phi, F) \right\rangle \nu(\dd z).
    \end{split}
  \end{equation*}
\end{lemma}
A similar formula holds for $P$:
\begin{lemma}\label{lemma:integration_by_parts_P}
  Under the same assumptions as for \cref{lemma:integration_by_parts_energy} with $q = d = 1$, it holds
  \begin{equation*}
    \begin{split}
      \dE \phi(F) \delta u &= \dE (\phi(F+1) - \phi(F)) \nu(u DF) \\
                           &+ \dE \nu(u P(\phi, F)).
    \end{split}
  \end{equation*}
\end{lemma}

\subsubsection{Proofs of the qualitative results}
\begin{proof}[Proof of {\cref{theorem:convergence_stable_quali}}]
  We first prove the theorem under \cref{condition:W,condition:S}.
  By \cref{condition:S}, we have that
  \begin{equation*}
    \dE F_{n} F_{n}^{T} = \dE {[u_{n}, DF_{n}]}_{\Gamma} \tto{} \dE S S^{T} < \infty.
  \end{equation*}
  So $(F_{n})$ is bounded in $\CL^{2}(\dP)$.
  Let $G \in \CG$.
  For all $n \in \dN$, we let $\xi_{n} = (F_{n}, G)$.
  Since $(F_{n})$ is bounded in $\CL^{2}(\dP)$, $(\xi_{n})$ is tight.
  We can extract a subsequence (still denoted $(\xi_{n})$) such that $(\xi_{n})$ converges in law to $(F_{\infty}, G)$.
  Let $\psi_{n}(\lambda) = \dE G \e^{i \langle \lambda, F_{n} \rangle}$, and $\psi_{\infty}(\lambda) = \dE G \e^{i \langle \lambda, F_{\infty} \rangle}$.
  By convergence in law, we have that, as $n \to \infty$, $(\psi_{n})$ converges uniformly to $\psi_{\infty}$.
  Since $(\xi_{n})$ is bounded in $\CL^{2}(\dP)$ it is also uniformly integrable, and we find that
  \begin{equation*}
    \nabla \psi_{n}(\lambda) = i \dE F_{n}G \e^{i \langle \lambda, F_{n} \rangle} \tto{} i \dE F_{\infty}G \e^{i \langle \lambda, F_{\infty} \rangle} = \nabla \psi_{\infty}(\lambda).
  \end{equation*}
  By \cref{lemma:integration_by_parts_energy}, \cref{equation:chain_rule_minus,equation:chain_rule_plus,equation:Mecke}, we can find $(R_{n})$ such that
\begin{equation*}
    |R_{n}| \leq \lambda^{2} \dE \int {|u_{n}(z)|} {|D_{z}^{+}F_{n}|}^{2} \nu(\dd z),
  \end{equation*}
  and
  \begin{equation*}
    \begin{split}
      \nabla \psi_{n}(\lambda) &= i \dE G {\left[u_{n}, D\left(\e^{i \lambda F_{n}}\right)\right]}_{\Gamma} + i \dE \e^{i \langle \lambda, F_{n} \rangle} {[u_{n}, DG]}_{\Gamma} \\
                               &= - \lambda \dE G \e^{i \langle \lambda, F_{n} \rangle} {[u_{n}, DF_{n}]}_{\Gamma} + i \dE \e^{i \langle \lambda, F_{n} \rangle} {[u_{n}, DG]}_{\Gamma} +  R_{n}.
    \end{split}
  \end{equation*}
  We thus see that \cref{condition:S,condition:R3,condition:W} imply that
  \begin{equation*}
    \nabla \psi_{n}(\lambda) \tto{} -\lambda \dE S S^{T} \e^{i \langle \lambda, F_{\infty} \rangle}.
  \end{equation*}
  All in all, we have proved that
  \begin{equation*}
    \frac{\dd}{\dd \lambda} \psi_{\infty}(\lambda) = i \dE G F_{\infty} \e^{i \langle \lambda, F_{\infty} \rangle} = - \lambda \dE G S S^{T} \e^{i \langle \lambda, F_{\infty} \rangle}.
  \end{equation*}
  Thus, we obtain the following differential equation for the conditional characteristic function:
  \begin{equation*}
    \frac{\dd}{\dd \lambda} \dE [\e^{i \lambda F_{\infty}} | \eta] = -\lambda SS^{T} \dE[ \e^{i \lambda F_{\infty}} | \eta ].
  \end{equation*}
  The only solution of this equation with $\psi(0) =1$ is the one given in \cref{equation:conditional_Fourier_Gaussian_mixture}  and this concludes the proof in view of \cref{condition:stable_convergence_G} in \cref{proposition:stable_convergence_G}.
  For the proof under \cref{condition:Wnu,condition:Snu}, we only briefly explain what to modify; the details can be found below, in the proof of \cref{theorem:convergence_stable_quali_P}, where we use this strategy to obtain convergence to a Poisson mixture.
  To work with \cref{condition:Wnu,condition:Snu}, we rather introduce $\psi_{n}(\lambda) = \dE \e^{i \langle \lambda, F_{n} + I_{1}(h) \rangle}$ for some $h \in \CL^{2}(\nu)$.
  Instead of \cref{lemma:integration_by_parts_energy}, we have to use \cref{lemma:integration_by_parts_D}, and we can use \cref{equation:chain_rule_plus} directly without invoking the Mecke formula \cref{equation:Mecke}.
  One concludes with \cref{condition:stable_convergence_I} in \cref{proposition:stable_convergence_G}.
  The rest of the proof is similar.
\end{proof}

\begin{proof}[Proof of {\cref{theorem:convergence_stable_quali_P}}]
  Let $h \in \CL^{2}(\nu)$.
  Let $\lambda \in \dR$, and consider $\psi_{n}(\lambda) = \dE \e^{i \lambda (F_{n} + I_{1}(h))}$, and $\psi_{\infty}(\lambda) = \dE \e^{i \lambda (F_{\infty} + I_{1}(h))}$.
  Since $\dE F_{n}^{2} = \dE \langle u_{n}, D F_{n} \rangle$, using \cref{condition:M}, we see that $F_{n} + I_{1}(h)$ is tight and uniformly integrable.
  Up to extraction, we can find some $F_{\infty}$, such that
  \begin{equation*}
    F_{n} + I_{1}(h) \tto{} F_{\infty} + I_{1}(h),
  \end{equation*}
  and that
  \begin{equation}\label{equation:phi_prime_P}
    \psi_{n}'(\lambda) = i \dE (F_{n} + I_{1}(h)) \e^{i \lambda (F_{n} + I_{1}(h))} \tto{} i \dE (F_{\infty} + I_{1}(h)) \e^{i \lambda (F_{\infty} + I_{1}(h))} = \psi_{\infty}'(\lambda).
  \end{equation}
  On the other hand, by \cref{lemma:integration_by_parts_P,equation:chain_rule_difference}, we have that
  \begin{equation*}
    \psi_{n}'(\lambda) = i \dE \langle u_{n}, DF_{n} + h \rangle\e^{i\lambda (F_{n}+I_{1}(h))}(\e^{i\lambda}-1) + i \dE I_{1}(h) \e^{i \lambda(F_{n}+I_{1}(h))} + R_{n},
  \end{equation*}
  where
  \begin{equation*}
    |R_{n}| \leq \lambda^{2} \int |u_{n}(z) (D_{z}F_{n}-1) D_{z}F_{n}| \nu(\dd z).
  \end{equation*}
  Thus, under \cref{condition:M,condition:P3,condition:W}:
  \begin{equation}\label{equation:phi_prime_P2}
    \lim_{n \to \infty} \psi_{n}'(\lambda) = i \dE (\e^{i\lambda} - 1)M \e^{i\lambda F_{\infty}} + i \dE I_{1}(h) \e^{i \lambda (F_{\infty}+I_{1}(h))}.
  \end{equation}
  Equating, \cref{equation:phi_prime_P,equation:phi_prime_P2}, we obtain that
  \begin{equation*}
    \dE \e^{i \lambda I_{1}(h)} F_{\infty} \e^{i \lambda F_{\infty}} = \dE \e^{i \lambda I_{1}(h)} M (\e^{i\lambda} -1) \e^{i \lambda F_{\infty}}, \quad \forall \lambda \in \dR,\, \forall h \in \CL^{2}(\nu).
  \end{equation*}
  Arguing, by linearity of $I_{1}$, as in the proof of \cref{condition:stable_convergence_I} of \cref{proposition:stable_convergence_G}, we find that:
  \begin{equation*}
    \dE I_{1}(h) F_{\infty} \e^{i\lambda F_{\infty}} = \dE I_{1}(h) M(\e^{i\lambda} -1) \e^{i \lambda F_{\infty}}, \quad \forall \lambda \in \dR,\, \forall h \in \CL^{2}(\nu).
  \end{equation*}
  That is to say, we have proved the following differential equation for the conditional characteristic function:
  \begin{equation*}
    \frac{\dd}{\dd \lambda} \dE [ \e^{i\lambda F_{\infty}} | \eta ] = i (\e^{i\lambda} - 1)M \dE [\e^{i \lambda F_{\infty}} | \eta].
  \end{equation*}
  The unique solution of this equation satisfying $\psi(0) = 1$ is the function given in \cref{equation:conditional_Fourier_Poisson_mixture}.
  This concludes the proof by \cref{condition:stable_convergence_I} in \cref{proposition:stable_convergence_G}.
\end{proof}

\subsubsection{Proofs of the quantitative results in the multivariate case}
\cref{theorem:convergence_stable_quantitative} follows from one of the two following generic bounds, namely either \cref{proposition:quantitative_bound_energy_D} with $h = 0$ and $\phi \in \CF_{3}$ for the case of ${[\cdot, \cdot]}_{\nu}$; or \cref{proposition:quantitative_bound_energy} with $G=1$ and $\phi \in \CF_{3}$ for the case of ${[\cdot, \cdot]}_{\Gamma}$.
We obtain these bounds via the so-called \emph{Talagrand's smart path interpolation} method.
For shot, given $\phi \colon \dR^{d} \to \dR$ smooth, we write $\Phi_{k} = \sup_{x \in \dR^{d}} |\nabla^{k} \phi|(x)$, for $k \in \{1,2,3\}$.
\begin{proposition}\label{proposition:quantitative_bound_energy_D}
  Let $F = (F_{1}, \dots, F_{d}) \in \dom D$, $S \in \cov$, and $N$ be a standard $d$-dimensional Gaussian vector independent of $\eta$.
  Assume that there exists $u \in \dom \delta$ such that $F = \delta u$.
  Then, for all $\phi \in \CC_{b}^{3}(\dR^{d})$ and all $I = I_{1}(h)$, $h \in \CL^{2}(\nu)$:
\begin{equation}\label{equation:quantitative_bound_energy_D}
  \begin{split}
    |\dE \phi(F + I) - \dE \phi(S N + I)| \leq & \frac{1}{2} \Phi_{2} \dE  {\left|{{\left[u, DF\right]}}_{\widetilde{\nu}} - S S^{T}\right|} \\
                                                             &+ \frac{2}{3} \Phi_{3} \dE {\left|{\left[u, (DS)S^{T}\right]}_{\nu}\right|} \\
                                                             &+ \Phi_{3} \dE \int {|h(z) + u(z)|} \left({|D_{z}^{+}F|}^{2} + {|D S|}^{2}\right) \nu(\dd z).
  \end{split}
\end{equation}
\end{proposition}
\begin{proposition}\label{proposition:quantitative_bound_energy}
  Let $F = (F_{1}, \dots, F_{d}) \in \dom D$, $S \in \cov$, and $N$ be a standard $d$-dimensional Gaussian vector independent of $\eta$.
  Assume that there exists $u \in \dom \delta$ such that $F = \delta u$.
Then, for all $\phi \in \CC_{b}^{3}(\dR^{d})$ and all $G \in \CG$:
\begin{equation}\label{equation:quantitative_bound_energy}
  \begin{split}
    |\dE \phi(F)G - \dE \phi(S N)G| &\leq  \frac{1}{2} \Phi_{2} {|G|}_{\infty} \dE  {\left|{{[u, DF]}}_{\widetilde{\Gamma}} - S S^{T}\right|} \\
                                        &+ \frac{2}{3} \Phi_{3} {|G|}_{\infty} \dE {|{[u, (DS)S^{T}]}_{\Gamma}|} \\
                                              &+ \Phi_{3} {|G|}_{\infty} \dE \int {|u(z)|} \left({|D_{z}^{+}F|}^{2} + {|D S|}^{2}\right) \nu(\dd z) \\
                                         &+ \Phi_{1} \dE {\left|{[u,DG]}_{\Gamma}\right|},
  \end{split}
\end{equation}
where for short, we write ${|G|}_{\infty} = {|G|}_{\CL^{\infty}(\dP)}$.
\end{proposition}
We start by proving in details the bounds involving ${[\cdot, \cdot]}_{\Gamma}$ that is more involved, then we explain how to adapt the proof for ${[\cdot, \cdot]}_{\nu}$.
\begin{proof}[Proof of {\cref{proposition:quantitative_bound_energy}}]
  By the assumptions on $u$ and $F$, we have that $\int {|u(z)|} {|D_{z}^{+}F|} \nu(\dd z) < \infty$, and we can assume that $\int {|u(z)|} {|D_{z}^{+}F|}^{2} \nu(\dd z) < \infty$ (otherwise there is nothing to prove).
  Let ${(s_{t})}_{t \in [0,1]}$ be a smooth $[0,1]$-valued path such that $s_{0} = 0$ and $s_{1} = 1$, and define
  \begin{equation*}
    F_{t} = s_{t}F+s_{1-t}S N.
  \end{equation*}
  Let $g(t) = \dE \phi(F_{t})G$.
  Then,
  \begin{equation*}
     \dE \phi(F)G - \dE \phi(S N)G = \int_{0}^{1} \dot{g}_{t} \dd t.
  \end{equation*}
  An explicit computation yields
  \begin{equation}\label{equation:g_dot_1}
    \dot{g}_{t} = \dE [ \langle \nabla\phi(F_{t}), (\dot{s}_{t}F - \dot{s}_{1-t}S N)G \rangle ].
  \end{equation}
  Since $\dom D$ is a linear space, in view of the assumptions, $F_{t} \in \dom D$.
    Since $\nabla \phi$ is Lipschitz, $\nabla\phi(F_{t}) \in \dom D$.
    Using the integration by part formula \cref{lemma:integration_by_parts_energy}, we find that
  \begin{equation}\label{equation:g_dot_1_partial}
\begin{split}
  \dE \langle \nabla \phi(F_{t})G, F \rangle &= \dE G \left\langle \nabla^{2} \phi(F_{t}), s_{t}{[u,DF]}_{\Gamma} + s_{1-t} {[u,(DS)N]}_{\Gamma}  \right\rangle \\
                                                        &+ \dE \left\langle \nabla \phi(F_{t}), {[u,DG]}_{\Gamma} \right\rangle \\
                                             &+ \frac{1}{2} \dE G \int \left\langle u(z), R_{z}^{+}(F_{t}, \nabla \phi) \right\rangle \nu(\dd z) \\
                                             &+ \frac{1}{2} \dE G \int \left\langle u^{-}(z), R_{z}^{-}(F_{t},\nabla \phi) \right\rangle \eta(\dd z).
    \end{split}
  \end{equation}
  Recall that, by integration by parts, $\dE N\psi(N) = \dE \nabla\psi(N)$, for all smooth $\psi$.
  Let
  \begin{equation*}
    \psi(x) = G\partial_{ij} \phi(s_{t}F+s_{1-t}S x).
  \end{equation*}
  Then,
  \begin{equation*}
  \partial_{k}\psi(x) = s_{1-t} G\sum_{l} S_{lk} \partial_{ijl}(s_{t}F+s_{1-t}S x).
  \end{equation*}
  As a consequence, by the previous Gaussian integration by parts:
  \begin{equation}\label{equation:gaussian_ipp_1}
    \dE G \langle \nabla^{2}\phi(F_{t}), {[u, (DS)N]}_{\Gamma} \rangle = s_{1-t} \dE G \langle \nabla^{3} \phi(F_{t}), {[u, (DS)S^{T}]}_{\Gamma}\rangle.
\end{equation}
  Furthermore, by Gaussian integration by parts, we obtain that
  \begin{equation}\label{equation:gaussian_ipp_2}
    \dE \langle \nabla\phi(F_{t}), S N \rangle = s_{1-t} \dE \langle \nabla^{2}\phi(F_{t}), S S^{T} \rangle.
  \end{equation}
  Combining \cref{equation:g_dot_1,equation:g_dot_1_partial,equation:gaussian_ipp_1,equation:gaussian_ipp_2}, we find that
  \begin{equation}\label{equation:g_dot_2}
    \begin{split}
      \dot{g}_{t} &= \dE G \left\langle \nabla^{2}\phi(F_{t}), \left( s_{t}\dot{s}_{t} {{[u, DF]}}_{\widetilde{\Gamma}} - s_{1-t}\dot{s}_{1-t}S S^{T} \right) \right\rangle \\
                  &+ \dot{s}_{t} s_{1-t}^{2} \dE G \langle \nabla^{3}\phi(F_{t}), {[u (DS)S^{T}]}_{\Gamma}\rangle \\
                  &+ \dot{s}_{t} \dE \left\langle \nabla \phi(F_{t}), {[u,DG]}_{\Gamma} \right\rangle \\
                  &+\dot{s}_{t} \frac{1}{2} \dE G \int \left\langle u(z), R_{z}^{+}(F_{t}, \nabla \phi) \right\rangle \nu(\dd z) \\
                  &+\dot{s}_{t} \frac{1}{2} \dE G \int \left\langle u^{-}(z), R_{z}^{-}(F_{t},\nabla \phi) \right\rangle \eta(\dd z).
    \end{split}
  \end{equation}
  Hence, by the Cauchy-Schwarz inequality, we find that
  \begin{equation*}
\begin{split}
  |\dot{g}_{t}| & \leq {|G|}_{\infty} \Phi_{2} s_{t} \dot{s}_{t} \dE \left\langle \nabla^{2}\phi(F_{t}), \left( s_{t}\dot{s}_{t} {{[u, DF]}}_{\widetilde{\Gamma}} - s_{1-t}\dot{s}_{1-t}S S^{T} \right) \right\rangle \\
                &+ \Phi_{3} \dot{s}_{t} s_{1-t}^{2} {|G|}_{\infty} \dE {|{[u, (DS)S^{T}]}_{\Gamma}|} \\
                &+ \Phi_{1} \dot{s}_{t} \dE {|{[u, DG]}_{\Gamma}|} \\
                &+ \frac{1}{2} \dot{s}_{t} {|G|}_{\infty} \Phi_{3} \dE \int {|u(z)|} {|D_{z}^{+}F_{t}|}^{2}  \nu(\dd z) \\
                &+ \frac{1}{2} \dot{s}_{t} {|G|}_{\infty} \Phi_{3} \dE \int {|u^{-}(z)|} {|D_{z}^{-}F_{t}|}^{2}  \nu(\dd z).
\end{split}
    \end{equation*}
    By the Mecke formula \cref{equation:Mecke}, the last two lines are equal.
    By expending the square in ${|D_{z}^{+}F_{t}|}^{2}$ and using that $N$ is centered and independent of $\eta$, the cross term vanishes in the expectation.
    By the fact that $N$ is a normal vector independent of $\eta$, we also find that $\dE [{|(D_{z}^{+}S)N|}^{2} | \eta] = {|(D_{z}^{+}S)|}^{2}$.
    Following these observations, the results is obtained by selecting $s_{t} = t^{\frac{1}{2}}$ (other choices of $s$ could possibly yield better constants).
  The reader can immediately verify that with this choice for $s$, we have that
  \begin{align*}
    & \int_{0}^{1} \dot{s}_{t} \dd t = 1;\\
    & \int_{0}^{1} s_{t} \dot{s}_{t} \dd t = \int_{0}^{1} s_{1-t} \dot{s}_{1-t} = \frac{1}{2}; \\
    & \int_{0}^{1} \dot{s}_{t} s^{2}_{1-t} = \int_{0}^{1} s_{t}^{2} \dot{s}_{t} = \frac{2}{3}.
  \end{align*}
  This concludes the proof.
\end{proof}

\begin{proof}[Proof of {\cref{proposition:quantitative_bound_energy_D}}]
  The strategy of proof is the same and we simply highlight the differences with the previous proof.
  We have to consider instead $g(t) = \dE \phi(F_{t} + I_{1}(h))$ for some $h \in \CL^{2}(\nu)$.
  Then, using \cref{lemma:integration_by_parts_D}, we find that
  \begin{equation*}
    \begin{split}
      \dE \langle \nabla \phi(F_{t} + I_{1}(h)), F \rangle &= \frac{1}{2} \dE G \left\langle \nabla^{2} \phi(F_{t}), s_{t}{[u,DF]}_{\nu} + s_{1-t} {[u,(DS)N]}_{\nu}  \right\rangle \\
                                                           &+ \frac{1}{4} \dE G \int \left\langle (u(z) + h(z)), R_{z}^{+}(F_{t}, \nabla \phi) \right\rangle \nu(\dd z).
    \end{split}
  \end{equation*}
  The rest of the proof is identical to the previous one.
\end{proof}

\begin{proof}[Proof of {\cref{theorem:convergence_stable}}]
  Let $h \in \CL^{2}(\nu)$, and $N \sim \lN(0,id_{\dR^{d}})$.
  For $n \in \dN$, we write:
  \begin{equation*}
    \begin{split}
      \dE \phi(F_{n} + I_{1}(h)) - \dE \phi(S N + I_{1}(h)) &= \dE \phi(F_{n} + I_{1}(h)) - \dE \phi(S_{n}N + I_{1}(h)) \\
                                                            &+ \dE \phi(S_{n}N + I_{1}(h)) - \dE \phi(SN + I_{1}(h)).
    \end{split}
  \end{equation*}
  From \cref{proposition:quantitative_bound_energy_D}, we have that under \cref{condition:Wnu,condition:R3,condition:epsilon,condition:RS,condition:Rh,condition:S3,condition:Sh}:
  \begin{equation*}
    \dE \phi(F_{n} + I_{1}(h)) - \dE \phi(S_{n}N + I_{1}(h)) \tto{} 0.
  \end{equation*}
  On the other hand, under \cref{condition:Cst}, $S_{n}N \tto{stably} S N$, consequently
  \begin{equation*}
    \dE \phi(S_{n}N + I_{1}(h)) - \dE \phi(SN + I_{1}(h)) \tto{} 0.
  \end{equation*}
  We conclude using \cref{proposition:stable_convergence_G}.
\end{proof}

\subsubsection{Proofs of the quantitative result in the univariate case}
In order to deduce \cref{theorem:stable_convergence_wasserstein} from \cref{proposition:quantitative_bound_energy_D}, we need a regularization lemma.
Results comparing the Wasserstein distance with an other variational distance are well-known to the experts, for completeness we state and prove a result here.
\cref{theorem:stable_convergence_wasserstein} is immediately deduced from \cref{proposition:quantitative_bound_energy_D} (with $h = 0$) and the following lemma.
\begin{lemma}\label{lemma:regularization_MK}
  Let $F$ and $F' \in \CL^{1}(\dP)$ such that there exists $a$, $b$, and $c \geq 0$ such that for all $\phi \in \CC^{3}_{b}(\dR)$:
  \begin{equation*}
    \dE \phi(F) - \dE \phi(F') \leq a {|\phi'|}_{\infty} + b {|\phi''|}_{\infty} + c {|\phi'''|}_{\infty}.
  \end{equation*}
  Then,
  \begin{equation}\label{equation:regularization_MK}
    d_{1}(F, F') \leq a + \max\left(\left(2^{\frac{1}{3}} + 2^{-\frac{2}{3}}\right) \Delta_{1}^{\frac{2}{3}} \Delta_{2}^{\frac{1}{3}}, \Delta_{1} \Delta_{2}\right)
  \end{equation}
  where
  \begin{align*}
    & \Delta_{1} = 2 {\left(\frac{2}{\pi}\right)}^{\frac{1}{2}} + \dE |F| + \dE |F'|; \\
    & \Delta_{2} = {\left(\frac{2}{\pi}\right)}^{\frac{1}{2}} b + 2^{\frac{1}{2}} c.
  \end{align*}
\end{lemma}

\begin{proof}
  This result is well-known at different levels of generality, and we follow here the proof of \cite[Theorem 3.4]{NourdinNualartPeccatiStableLimits} (where the reader is referred to for details).
  For $t \in (0,1)$, we define $\phi_{t}(x) = \int \phi(t^{\frac{1}{2}} y + {(1-t)}^{\frac{1}{2}} x) \gamma(\dd y)$, with $\gamma = \lN(0,1)$.
  Then, we have that
  \begin{align*}
    & {|\phi_{t}'|}_{\infty} \leq {|\phi'|}_{\infty}; \\
    & {|\phi_{t}''|}_{\infty} \leq {\left(\frac{2}{\pi}\right)}^{\frac{1}{2}} \frac{{|\phi'|}_{\infty}}{t}; \\
    & {|\phi_{t}'''|}_{\infty} \leq 2^{\frac{1}{2}} \frac{{|\phi'|}_{\infty}}{t}.
  \end{align*}
  On the other hand, we have that
  \begin{equation*}
    \dE \phi(F) - \dE \phi_{t}(F) \leq t^{\frac{1}{2}} {|\phi'|}_{\infty} \left( {\left(\frac{2}{\pi}\right)}^{\frac{1}{2}} + \dE |F|\right).
  \end{equation*}
  Combining all the estimates and optimizing in $t$ yields the desired result.
\end{proof}
\begin{remark}
  From the proof, we see that we do not expect that \cref{equation:regularization_MK} is optimal.
  Consequently, all estimates deduced from \cref{lemma:regularization_MK} are sub-optimal, in particular, \cref{theorem:stable_convergence_wasserstein} is a priori sub-optimal.
\end{remark}
\begin{proof}[Proof of {\cref{theorem:stable_convergence_wasserstein_sequence}}]
  By the triangle inequality, we write
  \begin{equation*}
    d_{1}(F_{n}, \lN(0,S^{2})) \leq d_{1}(F_{n}, \lN(0,S_{n}^{2})) + d_{1}(\lN(0,S_{n}^{2}), \lN(0,S^{2})).
  \end{equation*}
  By \cref{theorem:stable_convergence_wasserstein}, we have that
  \begin{equation} 
d_{1}(F_{n}, \lN(0,S_{n}^{2}))  \leq \max\left(\left(2^{\frac{1}{3}} + 2^{-\frac{2}{3}}\right) \Delta_{1,n}^{\frac{2}{3}} \Delta_{2,n}^{\frac{1}{3}}, \Delta_{1,n} \Delta_{2,n}\right).
\end{equation}
Thus, to conclude the proof we need to prove that
\begin{equation*}
  a_{n} := d_{1}(\lN(0,S_{n}^{2}), \lN(0,S^{2})) \leq {\left(\frac{2}{\pi}\right)}^{\frac{1}{2}} d_{1}(S_{n}, S).
\end{equation*}
Let $A_{n} \sim S_{n}$ and $A \sim S$.
Let $N \sim \lN(0,1)$ independent of $A$ and $A_{n}$.
Then $(A_{n}N, AN)$ is a coupling of $(\lN(0, S_{n}^{2}), \lN(0,S^{2}))$.
Hence, by the formulation of the Wasserstein distance as an infimum over couplings, we find that:
\begin{equation*}
  a_{n} \leq \dE |(A-A_{n})N| = {\left(\frac{2}{\pi}\right)}^{\frac{1}{2}} \dE |A - A_{n}|.
\end{equation*}
Minimizing over all couplings $(A, A_{n})$ proves the claim.
This completes the proof.
\end{proof}
\begin{remark}
  From the proof, we see that working with the Wasserstein distance is crucial.
  For instance, we do not know if $d_{3}(\lN(0,S^{2}), \lN(0,T^{2})) \leq c d_{3}(S, T)$, for some $c > 0$.
\end{remark}

\section{Convergence of stochastic integrals}%
\label{section:stochastic_integrals}

\subsection*{Outline}
We apply the results of \cref{section:abstract_results} to stochastic integrals.
In particular, we deduce \cref{proposition:stable_fourth_moment}, that is a stable version of the fourth moment theorem of \fcite{DoeblerPeccatiFMT}, and \fcite{DoeblerVidottoZheng}; and \cref{theorem:order_2_poisson,theorem:order_2_gaussian} that give sufficient conditions for a sequence of Itô-Poisson integrals of order $2$ to converge to a Gaussian or Poisson mixture.
We recall that a stochastic integral of order $q$ is simply an eigenvector of $L$ associated with the eigenvalues $q$.
More precisely, it is possible to construct a bijective isometry $I_{q}$ from the symmetric functions of $\CL^{2}(\nu^{q})$ to $\ker(L+q)$.
Hence every stochastic integral can be written $I_{q}(h)$ for some $h \in \CL^{2}(\nu^{q})$ symmetric.
Moreover the mapping $I_{q}$ is extended to a continuous mapping from $\CL^{2}(\nu^{q})$ to $\CL^{2}(\dP)$ (by setting $I_{q}(h) = I_{q}(\tilde{h})$ where $h$ is the symmetruzation of $h$).
Here we only use three properties: that $\ker(L+q) \subset \dom D$ and that $D_{z} I_{q}(h) = q I_{q-1}(h(z, \cdot))$; that $\dE {I_{q}(h)}^{2} = q! \nu^{q}(h^{2})$; as well as a product formula for stochastic integrals that expresses the product $I_{q}(h) I_{q'}(h')$ as a linear combination of stochastic integrals of order no greater than $p + q$ and whose itegrands can be written explicitely in terms of the so-called \emph{star-contractions} of $h$ and $h'$ (see \cite[Proposition 5]{LastAnaSto} as well as \cite{DoeblerPeccatiProduct}).
More details can be found on these stochastic integrals in \cite{ItoPoissonIntegrals,Surgailis,LastAnaSto,LastPenrose}.

\subsection{A stable fourth-moment theorem for normal approximation}%
\label{section:fourth_moment}
In a recent reference, \cite{DoeblerVidottoZheng} proves a multidimensional fourth-moment theorem on the Poisson space, thus refining and generalizing the previous findings of \cite{DoeblerPeccatiFMT}.
It is worth noting that taking $G = 0$ and $S$ deterministic in~\cref{equation:quantitative_bound_energy} yields the same bound as~\cite[Equation 4.2]{DoeblerPeccatiFMT}.
In fact, as a first application of~\cref{theorem:convergence_stable}, we deduce a stable fourth-moment theorem on the Poisson space.

\begin{proposition}[Stable fourth-moment theorem]\label{proposition:stable_fourth_moment}
  Let $q_{1}, \dots, q_{d} \in \dN$.
  For $n \geq 1$, let $(f^{i}_{n}) \subset \CL^{2}(\mu^{q_{i}})$ and let $F_{n} = (I_{q_{1}}(f^{1}_{n}), \dots, I_{q_{d}}(f^{d}_{n}))$.
  Assume that $(F_{n})$ is bounded in $\CL^{2}(\dP)$.
  Then, the following are equivalent:
  \begin{enumerate}[(i)]
    \item $F_{n}$ converges stably to a Gaussian vector.\label{condition:fmt_i}
    \item For all $i \in [d]$, $F_{n}^{i}$ converges in law to a Gaussian random variable.\label{condition:fmt_ii}
    \item For all $i \in [d]$, $\dE {\left(F_{n}^{i}\right)}^{4} - 3 {\left(\dE {\left(F_{n}^{i}\right)}^{2}\right)}^{2} \tto{} 0$.\label{condition:fmt_iii}
    \item $\Var \Gamma(L^{-1}F_{n}, F_{n}) \tto{} 0$.\label{condition:fmt_iv}
    \end{enumerate}
\end{proposition}

\begin{remark}
  If either of the conditions of the theorem is satisfied then, as $n \to \infty$, $\Gamma(L^{-1}F_{n}, F_{n}) \to \sigma \sigma^{T}$ in $\CL^{2}(\dP)$, where $\sigma$ is some deterministic matrix.
  The covariance of the limit Gaussian vector is $\sigma \sigma^{T}$.
\end{remark}
\begin{remark}
  \cref{proposition:stable_fourth_moment} is very close to~\cite[Theorem 2.22]{BourguinPeccatiPortmanteau}.
  However, one condition of their theorem requires that the norms of each of the individual star-contractions vanish.
  This is strictly stronger than a vanishing fourth-moment as, by the product formula, this condition translates in vanishing properly chosen linear combinations of the star-contractions (see \cite{DoeblerPeccatiProduct}).
\end{remark}

\begin{proof}
  It is clear that \cref{condition:fmt_i} implies \cref{condition:fmt_ii}.
  That \cref{condition:fmt_ii} implies (and in fact is equivalent to) \cref{condition:fmt_iii} is the main finding of~\cite{DoeblerPeccatiFMT}.
  That \cref{condition:fmt_iii} implies \cref{condition:fmt_iv} is a consequence of~\cite[Equation 4.3, Lemma 4.1]{DoeblerVidottoZheng}.
  Let us prove \cref{condition:fmt_iv} implies \cref{condition:fmt_i}.
  Under \cref{condition:fmt_iv}, $-\Gamma(L^{-1}F_{n}, F_{n}) \to \sigma \sigma ^{T}$, in $\CL^{2}(\dP)$, as $n \to \infty$.
  We apply \cref{theorem:convergence_stable_quali} with $S = \sigma$ and \begin{equation*}
    u_{n}(z) = - \left(I_{q_{1}-1}(f_{n}^{1}(z,\cdot)), \dots, I_{q_{d}-1}(f_{n}^{d}(z,\cdot))\right)= - D_{z} L^{-1}F_{n}.
  \end{equation*}
  We have that $\delta u_{n} = - \delta D L^{-1} F_{n} = F_{n}$, and that ${[u_{n}, F_{n}]}_{\Gamma} = - \Gamma(L^{-1}F_{n}, F_{n})$.
  Thus, \cref{condition:S} is satisfied.
From~\cite[Lemma 3.4]{LastPeccatiSchulteSecondOrderPoincare}, we have that
\begin{equation*}
  \label{equation:continuity_pseudo_inverse}
\dE \int {|D_{z}L^{-1}F_{n}|}^{2} \nu(\dd z) \leq \dE \int {|D_{z}F_{n}|}^{2} \nu(\dd z).
\end{equation*}
Hence, applying the Cauchy-Schwarz inequality, we find that
\begin{equation*}
  \dE \int {|u_{n}(z)|} {|D_{z}F_{n}|}^{2} \nu(\dd z) \leq \sqrt{\dE \int {|D_{z}F_{n}|}^{2} \nu(\dd z) \dE \int {|D_{z}F_{n}|}^{4} \nu(\dd z)}.
\end{equation*}
By Hölder's inequality, we find that (recall $DG \in \CL^{\infty}(\dP \otimes \nu)$ by \cref{lemma:D_bounded}):
\begin{equation*}
  \dE |[u_{n}, DG]| \leq {|DG|}_{\infty} {\left(\dE \int {|DF_{n}|}^{4}\right)}^{1/4}.
\end{equation*}
The quantity
\begin{equation*}
  \dE \int {|D_{z}F_{n}|}^{2} \nu(\dd z) = \sum_{i \in [d]} q_{i}! \nu^{q_{i}}({|f_{n}^{i}|}^{2}),
\end{equation*}
is bounded by assumption.
Hence it is sufficient to show that under \cref{condition:fmt_iv},
\begin{equation*}
  \dE \int {|D_{z}F_{n}|}^{4} \nu(\dd z) \tto{} 0.
\end{equation*}
This follows from \cite[Lemma 3.2]{DoeblerPeccatiFMT} and \cite[Remark 5.2]{DoeblerVidottoZheng}.
  The proof is complete.
\end{proof}

\begin{remark}
  Let $\Sigma = (\Sigma_{ij})$ be a random matrix sufficiently integrable.
  If we assume that 
  \begin{align}
    & -\Gamma(L^{-1}F_{n}, F_{n}) \tto{\CL^{1}(\dP)} \Sigma \Sigma^{T};\\
    & \dE {(F^{i}_{n})}^{4} \tto{} 3 {\left(\dE {\left(\Sigma \Sigma^{T}\right)}_{ii}\right)}^{2},\, \forall i = 1,\dots, d.
  \end{align}
  Then, from the previous computations, $\Sigma = \sigma$ is deterministic.
  This shows that fourth-moment theorems cannot capture phenomena with asymptotic random variances.
\end{remark}
\subsection{Convergence of order $2$ Poisson-Wiener integrals to a mixture}\label{section:order_2}

We derive an analytic statement for the convergence of a sequence of random variables of the form $F = I_{2}(g)$ for some $g \in \CL_{\sigma}^{2}(\nu^{2})$.
We formulate our bound in terms of the so-called \emph{star-contractions} that naturally appears when considering products of stochastic integrals.
Here, let us define for $f$ and $g \in \CL^{2}(\nu^{2}))$:
\begin{align*}
  & f \star_{1}^{1} g(x,y) = \int f(x,z) g(y,z) \nu(\dd z) \\
  & f \star_{2}^{1} g(x) = \int f(x,y) g(x,y) \nu(\dd y).
\end{align*}
When $F = I_{2}(g)$, with $u_{0}(z) = -D_{z}L^{-1}F = I_{1}(g(z,\cdot))$, we have $F = \delta u_{0}$.
However, for every $\hat{g} \in \CL^{2}(\nu^{2})$ such that the symmetrization of $\hat{g}$ is $g$, we also have that $u(z) = I_{1}(\hat{g}(z, \cdot))$ is a solution to $\delta u = F$.
Having made this observation, we can thus specify our \cref{theorem:convergence_stable_quali,theorem:convergence_stable_quali_P} to the particular case where $F$ is an Poisson-Wiener stochastic integral of order $2$.
\begin{theorem}\label{theorem:order_2_gaussian}
  Consider the sequence of random variables  $\{F_{n} = I_{2}(g_{n}); n \in \dN\}$ for some $(g_{n}) \subset \CL^{2}_{\sigma}(\nu^{2})$.
  Suppose that there exists $(\hat{g}_{n}) \subset \CL^{2}(\nu^{2})$ such that, for all $n \in \dN$ the symmetrization of $\hat{g}_{n}$ is $g_{n}$.
  Assume, moreover, that:
  \begin{align*}
    & \begin{cases}
      & g_{n} \star_{1}^{1} \hat{g}_{n} \tto{\CL^{2}(\nu^{2})} g_{2,\infty},\\
      & \nu^{2}(g_{n} \hat{g}_{n}) \tto{} g_{0,\infty};
  \end{cases}\tag{KS}\label{condition:Sg}
    \\
    & g_{n} \tto{\CL^{4}(\nu^{2})} 0; \tag{KR$_{4}$} \label{condition:R4g} \\
    & g_{n} \star_{2}^{1} g_{n} \tto{\CL^{2}(\nu)} 0; \tag{KR$_{\star}$} \label{condition:Rstar} \\
    & \hat{g}_{n} \star_{1}^{1} h \tto{\CL^{2}(\nu)} 0 \tag{KW}\label{condition:Wg}.
  \end{align*}
  Assume that $S^{2} = I_{2}(g_{2,\infty}) + g_{0,\infty} \geq 0$.
  Then $F_{n} \tto{stably} \lN(0,S^{2})$.
\end{theorem}
\begin{remark}
  Following \cref{theorem:stable_convergence_wasserstein}, it is of course possible to write conditions for the convergence in the Wasserstein distance $d_{1}$ in terms of the norms of the kernels.
  However, this task seems tedious and not particularly useful in this abstract setting.
\end{remark}
\begin{theorem}\label{theorem:order_2_poisson}
  Consider the sequence of random variables  $\{F_{n} = I_{2}(g_{n}); n \in \dN\}$ for some $(g_{n}) \subset \CL^{2}_{\sigma}(\nu^{2})$.
  Suppose that there exists $(\hat{g}_{n}) \subset \CL^{2}(\nu^{2})$ such that, for all $n \in \dN$ the symmetrization of $\hat{g}_{n}$ is $g_{n}$.
  Assume that \cref{condition:Sg,condition:Wg,condition:Rstar} hold and that
  \begin{equation*}
    \tag{KP$_{4}$}\label{condition:P4g}
      \nu^{2}\left(g_{n}^{2}{\left(g_{n}-\frac{1}{2}\right)}^{2}\right) \tto{} 0.
  \end{equation*}
  Assume that $M = I_{2}(g_{2,\infty}) + g_{0,\infty} \geq 0$.
  Then $F_{n} \tto{stably} \lP\lo(M)$.
\end{theorem}

\begin{proof}[Proof of {\cref{theorem:order_2_gaussian,theorem:order_2_poisson}}]
We prove the two theorems at once.
We simply apply \cref{theorem:convergence_stable_quali,theorem:convergence_stable_quali_P} to our data.
For simplicity, we drop the dependence in $n$.
Let $u = I_{1}(\hat{g})$.
Let us compute ${[DF, u]}_{\nu} = \nu(DF u)$ in that case.
  By the product formula \cite[Proposition 5]{LastAnaSto}, we have that
\begin{equation*}
  \begin{split} 
    {[DF, u]}_{\nu} &= \nu(u DF) = \int I_{1}(g(z,\cdot)) I_{1}(\hat{g}(z,\cdot)) \nu(\dd z) \\
                    &= \int I_{2}(g(z,\cdot) \otimes \hat{g}(z,\cdot)) + I_{1}(g(z,\cdot)\hat{g}(z,\cdot)) + \nu(g(z,\cdot)\hat{g}(z,\cdot)).
  \end{split}
\end{equation*}
By linearity of $I_{1}$ and $I_{2}$, we thus find 
\begin{equation*}
  \nu(uDF) = I_{2}(g \star_{1}^{1} \hat{g}) + I_{1}(g \star_{2}^{1} \hat{g}) + \nu^{2}(g\hat{g}).
\end{equation*}
By \cite[Lemma 2.4 (vi)]{DoeblerPeccatiUStats} (which according to the proof holds for any $\sigma$-finite measure $\nu$), we have that $\nu({|g \star_{2}^{1} \hat{g}|}^{2}) \leq {\nu({|g \star_{2}^{1} g|}^{2})}^{\frac{1}{2}} {\nu({|\hat{g} \star_{2}^{1} \hat{g}|}^{2})}^{\frac{1}{2}}$.
Hence, we see that \cref{condition:Sg,condition:Rstar} implies, by the continuity of $I_{2} \colon \CL^{2}(\nu^{2}) \to \CL^{2}(\dP)$  either \cref{condition:Snu} or \cref{condition:M} with $S^{2}$ or $M$ as given in the statement.
On the one hand, we have that
\begin{equation*}
  \begin{split}
    \frac{1}{16} \dE \int {|D_{z}F|}^{4} \nu(\dd z) &= \int \dE {I_{1}(g(z,\cdot))}^{4} \nu(\dd z) \\
                                       &= 3 \int {\left(\int g(y,z)^{2} \nu(\dd y)\right)}^{2} \nu(\dd z) + \int \int {g(y,z)}^{4} \nu(\dd y) \nu(\dd z) \\
                                                    &= 3 \nu\left({(g \star_{2}^{1} g)}^{2}\right) + \nu^{2}(g^{4}).
  \end{split}
\end{equation*}
To obtain the second equality, we use that $I_{1}(g(\cdot, z))$ has the same law as $\lP\lo(\nu(g(\cdot, z)))$; and we obtain the last equality by some easy algebraic manipulations.
So that \cref{condition:R4g,condition:Rstar} readily implies \cref{condition:R4}.
On the other hand:
\begin{equation*}
  \frac{1}{16} \dE \int {|D_{z}F(D_{z}F-1)|}^{2} \nu(\dd z) = \nu^{2}\left(g^{2}{\left(g-\frac{1}{2}\right)}^{2}\right) + 3 \nu\left({(g \star_{2}^{1} g)}^{2}\right).
\end{equation*}
We thus see that \cref{condition:P4g,condition:Rstar} implies \cref{condition:P4}.
Finally, we find that
\begin{equation*}
  \nu(u h) = I_{1}(\hat{g} \star_{1}^{1} h).
\end{equation*}
Consequently, using the continuity of $I_{1}$, we find that \cref{condition:Wg} implies \cref{condition:Wnu}.
\end{proof}

\section{Convergence of a quadratic functional of a Poisson process on the line}\label{section:application}

In this section, we apply our abstract result to show a limit theorem for a particular quadratic functional.
Let us recall one of the main applications of \cite{NourdinNualartLimitsSkorokhod,NourdinNualartPeccatiStableLimits}, refining a result of \cite{PeccatiYor}.
\begin{theorem}[{\cite[Example 4.2]{NourdinNualartLimitsSkorokhod}} and {\cite[Theorem 3.7]{NourdinNualartPeccatiStableLimits}}]\label{theorem:NNP}
  Let $W$ be a standard Brownian motion on $[0,1]$ and let
  \begin{equation*}
    Q_{n} = \frac{n^{\frac{3}{2}}}{\sqrt{2}} \int_{0}^{1} t^{n-1} (W_{1}^{2} - W_{t}^{2}) \dd t, \quad n \in \dN.
  \end{equation*}
  Then,
  \begin{equation*}
    Q_{n} \tto{stably} \lN\left(0, W_{1}^{2}\right).
  \end{equation*}
  Moreover, there exists $c > 0$ such that, for all $n \in \dN$:
  \begin{equation*}
    d_{1}\left(Q_{n}, \lN\left(0, W_{1}^{2}\right)\right) \leq c n^{-\frac{1}{6}}.
  \end{equation*}
\end{theorem}
Let $\eta$ be a Poisson point process on $\dR_{+}$ with intensity the Lebesgue measure; and $\hat{N}_{t} = \eta([0,t]) - t$, for $t \in \dR$.
The process $\hat{N}$ is a martingale called a \emph{compensated Poisson process on the line}.
Recall that from \fcite{DynkinMandelbaum}, we have that
\begin{equation*}
  \left\{n^{-\frac{1}{2}} \hat{N}_{nt};\; t \geq 0 \right\} \tto{} W,
\end{equation*}
where the convergence holds in the sense of finite-dimensional distributions and in a stronger sense that we do not detail here.
Having made this remark the following thermo-dynamical limit appears as a natural generalization of \cref{theorem:NNP}.
\begin{theorem}\label{theorem:quadratic}
  Let 
  \begin{equation*}
    Q_{n} = \frac{n^{\frac{3}{2}}}{\sqrt{2}} \int_{0}^{1} t^{n-1} \left({\left(n^{-\frac{1}{2}}\hat{N}_{n}\right)}^{2} - {\left(n^{-\frac{1}{2}}\hat{N}_{nt}\right)}^{2}\right) \dd t, \quad n \in \dN.
  \end{equation*}
  Then,
  \begin{equation*}
    Q_{n} \tto{stably} \lN\left(0,W_{1}^{2}\right).
  \end{equation*}
\end{theorem}

\begin{proof}
  By Itô's formula (see, for instance \cite[Chapter I, Theorem 4.57]{JacodShiryaev}), we have that
  \begin{equation*}
    \hat{N}_{t}^{2} = 2 \int_{0}^{t} \hat{N}_{s^{-}} \dd \hat{N}_{s} + \sum_{s \leq t} {\left(\hat{N}_{s} - \hat{N}_{s^{-}}\right)}^{2}.
  \end{equation*}
  Since, a Poisson process only has jumps of size $1$, we find that
  \begin{equation*}
    \hat{N}^{2}_{t} = 2 \int_{0}^{t} \hat{N}_{s^{-}} \dd \hat{N}_{s} + N_{t}.
  \end{equation*}
  Hence, we can write
  \begin{equation*}
    Q_{n} = \sqrt{2} F_{n} + H_{n},
  \end{equation*}
  where
  \begin{align}
    & F_{n} =  n^{\frac{1}{2}} \int_{0}^{1} t^{n-1} \int_{nt}^{n} \hat{N}_{s^{-}} \dd \hat{N}_{s} \dd t; \\
    & H_{n} = {\left(\frac{n}{2}\right)}^{\frac{1}{2}}\int_{0}^{1} t^{n-1} (N_{n} - N_{nt}) \dd t.
  \end{align}
  Recalling that $N$ is a non-decreasing process and that $\dE N_{t} = t$, we find that
  \begin{equation*}
    \begin{split}
      \dE |H_{n}| &= 2^{-\frac{1}{2}} n^{\frac{3}{2}} \int_{0}^{1} t^{n-1} (1-t) \dd t \\
                  &= 2^{-\frac{1}{2}} n^{\frac{3}{2}} \left(\frac{1}{n} - \frac{1}{n+1}\right) \\
                  &= 2^{-\frac{1}{2}} \frac{n^{\frac{1}{2}}}{n+1} = O(n^{-\frac{1}{2}}).
    \end{split}
  \end{equation*}
  Consequently, in order to obtain the conclusions of the theorem for $(Q_{n})$ it suffices to obtain them for $(F_{n})$.
  By inverting the order of integration, we find:
  \begin{equation*}
    F_{n} =  n^{-\frac{1}{2}} \int_{0}^{n} \hat{N}_{s^{-}} {\left(\frac{s}{n}\right)}^{n} \dd \hat{N}_{s} = n^{-\frac{1}{2}-n} \int_{0}^{n} \hat{N}_{s^{-}} s^{n} \dd \hat{N}_{s} = \delta u_{n},
  \end{equation*}
  where
  \begin{equation*}
    u_{n}(s) = n^{-\frac{1}{2}} \hat{N}_{s} 1_{[0,n]}(s) {\left(\frac{s}{n}\right)}^{n}.
  \end{equation*}
  We have that
  \begin{equation*}
    F_{n} = n^{-\frac{1}{2}} \hat{N}_{n} n^{-n} \int_{0}^{n} s^{n} \dd \hat{N}_{s} + n^{-\frac{1}{2} - n} \int_{0}^{n} (\hat{N}_{s^{-}} - \hat{N}_{n}) s^{n} \dd \hat{N}_{s}.
  \end{equation*}
  Now observe that, by Skorokhod's isometry:
  \begin{equation*}
    \begin{split}
      \dE {\left(n^{-\frac{1}{2} - n} \int_{0}^{n} (\hat{N}_{s^{-}} - \hat{N}_{n}) s^{n} \dd \hat{N}_{s}\right)}^{2} &= n^{-2n-1} \int_{0}^{n} {(s-n)}^{2} s^{2n} \dd s \\
                                                                                                                     &= \frac{n}{(2n+1)(2n+2)} = O(n^{-1}).
    \end{split}
  \end{equation*}
  We have that
  \begin{align}
    & \dE {\left(n^{-\frac{1}{2}} \hat{N}_{n}\right)}^{2} = 1; \\
    & \dE {\left(n^{-\frac{1}{2}} \hat{N}_{n}\right)}^{4} = 3 + n^{-1}; \\
    & \dE {\left(n^{-n} \int_{0}^{n} s^{n} \dd \hat{N}_{s}\right)}^{2} = \frac{n}{2n+1}; \\
    & \dE {\left(n^{-n} \int_{0}^{n} s^{n} \dd \hat{N}_{s}\right)}^{4} = \frac{n}{4n+1} + 3 \frac{n^{2}}{{(2n+1)}^{2}}.
  \end{align}
  By our stable fourth moment theorem \cref{proposition:stable_fourth_moment}, we immediately find that:
  \begin{equation*}
  \left(n^{-\frac{1}{2}} \hat{N}_{n}, n^{-n} \int_{0}^{n} s^{n} \dd \hat{N}_{s}\right) \tto{stably} \lN\left(0, \begin{pmatrix} 1 & 0 \\ 0 & \frac{1}{2} \end{pmatrix}\right).
  \end{equation*}
  Thus proving that $\sqrt{2} F_{n} \tto{stably} \lN(0, W_{1}^{2})$, and hence that $Q_{n} \tto{stably} \lN(0, W_{1}^{2})$.
\end{proof}
\begin{remark}
  Rather than studying $\delta u_{n}$ with \cref{theorem:convergence_stable}, we simplify the problem by studying the convergence of two Itô-Wiener integrals.
  In fact, in our example, \cref{condition:R4} is not satisfied.
  With the notations of the proof, we have that $D_{s}F_{n} = n^{-n-\frac{1}{2}} \int_{0}^{n} {(s \vee t)}^{n} \dd \hat{N}_{t}$.
  An easy computation yields that
  \[
    \int_{0}^{n} \dE {(D_{s}F_{n})}^{4} \dd s \tto{} \frac{1}{4}.
  \]
  We do not know if \cref{condition:R3} holds, so we do not know if we could use \cref{theorem:convergence_stable} directly (or even invoke \cref{theorem:stable_convergence_wasserstein_sequence} to get a quantitative estimate).
\end{remark}

\section{Some open questions}
\begin{itemize}
  \item As already mentioned, we are interested in understanding which techniques we should consider to reach quantitative estimates for the convergence to a Poisson mixture.
  \item According to \cite[Remark 3.3 (b)]{NourdinNualartPeccatiStableLimits}, the results of \cite{NourdinNualartLimitsSkorokhod} can be understood as a variant of the \emph{asymptotic Knight theorem} about the convergence of Brownian martingales.
In the Poisson setting, it would be interesting to know if our results can be put in contrast with a corresponding martingale result.
  \item Very commonly, quantitative limit theorems in stochastic geometry rely on Malliavin-Stein bounds on the Poisson space (see among others \cite{ReitznerSchulteCLTUStatistics,LachiezeReyPeccatiContractions,LachiezeReyPeccatiMarkedProcess,PeccatiChenStein}).
    In particular, counting statistics of a nice class of rescaled geometric random graphs constructed from a Poisson point process exhibit a Gaussian or Poisson asymptotic behaviour depending on the regime of the rescaling.
    In view of our results, we ask whether it is possible to consider a wider class of geometric random graphs (including the previous one) whose counting statistics exhibit a convergence to a mixture.
\end{itemize}

\printbibliography

\end{document}